\documentclass[oneside,reqno]{amsart}
\usepackage{amssymb,amsmath}
\usepackage[english]{babel}
\usepackage{amsfonts}
\usepackage{color}
\usepackage{amsthm}
\usepackage[colorlinks=true,linkcolor=NavyBlue, citecolor=NavyBlue,urlcolor=NavyBlue]{hyperref}
\usepackage{graphicx}
\usepackage{mathtools}
\usepackage[usenames,dvipsnames]{pstricks}
\usepackage{bm}
\usepackage{amsmath}
\usepackage{listings}
\usepackage{hyperref}
\usepackage[usenames,dvipsnames]{pstricks}
 \usepackage{epsfig}
 \usepackage{pst-grad} % For gradients
 \usepackage{pst-plot} % For axes

\usepackage{mathtools}
\usepackage{orcidlink}
\usepackage{yfonts}

\usepackage{xy}
\usepackage[draft]{fixme}
\newcommand\deq{\mathrel{\stackrel{\makebox[0pt]{\mbox{\normalfont\tiny def}}}{=}}}

\theoremstyle{plain}
\newtheorem{theorem}{Theorem}[section]

\newtheorem{lemma}[theorem]{Lemma}
\newtheorem{proposition}[theorem]{Proposition}
\newtheorem{corollary}[theorem]{Corollary}

\theoremstyle{remark}
\newtheorem{remark}{Remark}

\theoremstyle{definition}
\newtheorem{definition}[theorem]{Definition}
\newtheorem{example}[theorem]{Example}
\newtheorem*{notation*}{Notation}
\usepackage{listings}
\lstdefinelanguage{GAP}{%
 morekeywords={%
 Assert,Info,IsBound,QUIT,%
 TryNextMethod,Unbind,and,break,%
 continue,do,elif,%
 else,end,false,fi,for,%
 function,if,in,local,%
 mod,not,od,or,%
 quit,rec,repeat,return,%
 then,true,until,while%
 },%
 sensitive,%
 morecomment=[l]\#,%
 morestring=[b]",%
 morestring=[b]',%
}[keywords,comments,strings]

\usepackage[T1]{fontenc}
\usepackage[variablett]{lmodern}
\usepackage{xcolor}
\usepackage[foot]{amsaddr}
\usepackage{enumerate}
\usepackage{arydshln}
 \usepackage{tikz}
\usepackage{lscape} 
\let \l \lambda

\newcommand{\listP}{\l = (\l_1,\l_2, \dots, \l_t)}

\newcommand{\mup}[1]{\widetilde{\mathcal{U}}_{#1}}

\newcommand{\wtp}{\widetilde{\pi}_n}

\newcommand{\dist}{\mathbb D}

\usepackage{stackengine}

\usepackage{cleveref}
\usepackage[draft]{fixme}
\usepackage{hyperref}

\let \l \lambda

\usepackage{multirow}
\newcommand{\NSg}{\mathrm{NSg}}
\newcommand{\NS}{\mathrm{NS}}

\newcommand{\KN}{\mathop{\mathrm{KN}}\nolimits}
\newcommand{\dbar}{{\overline{\mathbb{D}}}^{\, \text{o}}_{2k+2}}
\usetikzlibrary{patterns}

\begin{document}
\title[Geometric Characterization of Maximal Unrefinable Partitions]{A Geometric Characterization of Maximal Unrefinable Partitions via the Keith-Nath Transformation and Young Diagrams}

%On the Structure of Maximal Unrefinable Partitions and Their Correspondence with Distinct-Parts Partitions via Young Diagrams
%Maximal Unrefinable Partitions and Young Diagrams

%Young Diagram Symmetries and Combinatorial Proofs for Maximal Unrefinable Partitions
%Symmetries of Maximal Unrefinable Partitions
 
\author[R.~Aragona]{Riccardo Aragona
\orcidlink{0000-0001-8834-4358}
}
\author[L.~Campioni]{Lorenzo Campioni
\orcidlink{0009-0001-7328-2892}
}
\author[R.~Civino]{Roberto Civino 
\orcidlink{0000-0003-3672-8485}
}

\address{DISIM \\
 Universit\`a degli Studi dell'Aquila\\
 via Vetoio\\
 67100 Coppito (AQ)\\
 Italy}       

\email{riccardo.aragona@univaq.it}
\email{lorenzo.campioni1@univaq.it}
\email{roberto.civino@univaq.it}

\date{} \thanks{R. Aragona and R. Civino are members of INdAM-GNSAGA
 (Italy). R. Civino is funded by the Centre of excellence
 ExEMERGE at the University of L'Aquila. The authors gratefully acknowledge financial support from MUR-Italy through PRIN 2022RFAZCJ ``Algebraic Methods in Cryptanalysis'', with full funding provided for L.~Campioni.
}

\subjclass[2010]{11P81, 05A17, 05A19} \keywords{Unrefinable partitions, partitions into distinct parts, numerical semigroups, Young diagrams, Keith and Nath transformation, hook lengths}

\begin{abstract}
We investigate the combinatorial structure of unrefinable partitions through their correspondence with numerical sets and Young diagrams. Building on the bijection introduced by Keith and Nath, we apply a general geometric criterion that links the unrefinability of a partition directly to the hook lengths of its associated Young diagram. This criterion provides a structural method for the characterization of any unrefinable partition.

Using this general framework, we revisit the correspondence results between \emph{maximal} unrefinable partitions and partitions into distinct parts, previously established using enumerative methods. We provide alternative and purely combinatorial proofs of these bijections, focusing on the rigid symmetry structures of the Young diagrams. In the triangular weight case, we show that the corresponding diagrams are quasi-symmetric, i.e.\ symmetric up to a single extra column. We extend this analysis to the nontriangular case, showing that the diagrams either exhibit this same quasi-symmetric structure or are perfectly self-conjugate, depending on the maximal part. 
\end{abstract}

\maketitle

%%%%%%%%%%%%%%%%%%%%%%%%%%%%%%%%%%%%%%%%%%%%%%%%%%%%%%%%%%%%%%
%%%%%%%%%%%%%%%%%%%%%%%%%%%%%%%%%%%%%%%%%%%%%%%%%%%%%%%%%%%%%%
%%%%%%%%%%%%%%      S E C T I O N 1      %%%%%%%%%%%%%%%%%%%%%%%
%%%%%%%%%%%%%%%%%%%%%%%%%%%%%%%%%%%%%%%%%%%%%%%%%%%%%%%%%%%%%%
%%%%%%%%%%%%%%%%%%%%%%%%%%%%%%%%%%%%%%%%%%%%%%%%%%%%%%%%%%%%%%
\section{Introduction}

An integer partition $\lambda$ into distinct parts is said to be \emph{unrefinable} if no part $a\in\lambda$ can be replaced by a collection of distinct positive integers $b_1,\dots,b_m$ with $m\ge 2$ such that
\begin{enumerate}
    \item none of the $b_i$ belongs to $\lambda$,
    \item $a = b_1+\cdots+b_m$,
\end{enumerate}
and the resulting multiset
\[
(\lambda\setminus\{a\})\cup\{b_1,\dots,b_m\}
\]
is again a partition into distinct parts of the same integer. Equivalently, $\lambda$ is unrefinable if every part $a\in\lambda$ admits no decomposition into a sum of distinct positive integers all of which are absent from $\lambda$, and it is \emph{refinable} otherwise.

This special restriction between the parts naturally imposes an upper bound on the largest part that such a partition can contain. This maximal bound has been determined in previous works~\cite{aragona2022maximal,aragona2022number}, distinguishing if the weight of the partition is a triangular number or it is not.

Maximal unrefinable partitions, namely those reaching the upper bound for the largest part, have been completely classified using enumerative methods. It has been further shown that these maximal partitions are in bijective correspondence with certain subsets of partitions into distinct parts. Specifically, for triangular weights $T_n$ with $n$ odd, the number of such partitions was shown to be equal to the number of partitions of $\frac{n+1}{2}$ into distinct parts, whereas for even $n$ there is only one such partition. Analogous correspondences hold in the nontriangular case of weight $T_n - d$, where maximal unrefinable partitions are mapped to specific subsets of partitions into distinct parts depending on the parity of $n-d$ and the value of the maximal part.

While the underlying combinatorial mechanism is elementary, it is obscured by lengthy and involved computations.
The notion of unrefinability, however, reveals a rigid internal structure that is more naturally captured through a geometric interpretation.

In this paper, in light of the transformation introduced by Keith and Nath~\cite{keith2011partitions} (the KN transformation), we interpret unrefinable partitions through the lens of numerical sets and their associated Young diagrams. This perspective allows us to visually encode combinatorial properties of the partitions, providing a framework to study unrefinability in terms of the structure of the corresponding numerical set and the associated Young diagram.
 This methodology is based on  a recent general criterion for unrefinability that can be directly deduced from geometric properties of the Young diagram associated with a numerical set via the KN transformation~\cite{lorenzo}. This criterion provides a visual and structural way to determine whether a partition is unrefinable, independently of the specific enumeration of parts. 
  
 Our approach yields two main contributions. First, we apply this framework to the case of maximal unrefinable partitions of triangular weight $T_n$. For $n$ odd, we prove that the Young diagrams associated with these partitions exhibit a striking ``quasi-symmetry''. Specifically, they possess a symmetric staircase structure perturbed only by a single ``extra column'' appearing immediately after the main diagonal. We show that the hook lengths within this extra column completely encode the combinatorial data of the partition, allowing us to construct a bijective map to partitions of $k$ into distinct parts.

Second, we extend this geometric analysis to the nontriangular case of weight $T_n - d$. We identify that the geometric structure of these partitions is dictated by their maximal part $\lambda_t$. When $n-d$ is even, (which corresponds to partitions having maximal part $\l_t=2n-5$) we show that the associated Young diagrams are {self-conjugate}. When $n-d$ is odd ($\l_t=2n-4$), the diagrams exhibit the same quasi-symmetric pattern found in the triangular case, i.e.\ a symmetric staircase structure with a single extra column. Using these observations, we obtain constructive proofs of the correspondence between these partitions and their respective counterparts (partitions into distinct odd parts or standard distinct parts).

By applying this approach, we are able to recover the previously established results on maximal unrefinable partitions, which were originally obtained through enumerative methods~\cite{aragona2022maximal,aragona2022number}. Beyond merely reproducing these known facts, our framework highlights the combinatorial mechanisms underlying unrefinable partitions and clarifies their deep connection with numerical sets and Young diagrams. In this sense, the criterion derived from the Young diagram not only serves as a tool for verification but also offers a unified perspective linking the structural properties of partitions to their enumerative behavior.

\subsection{Related works}
Unrefinable partitions were first introduced in a group-theoretic context, where they arise in a natural and surprisingly rigid correspondence with the generators of a normalizer chain of subgroups~\cite{aragona2021unrefinable}. This initial appearance links the combinatorial notion of unrefinability with structural properties of finite groups and motivated subsequent investigations into their combinatorial and enumerative behavior.

The computational analysis of unrefinable partitions was addressed in~\cite{aragona2023verification}, where the authors introduced an efficient and general algorithm capable of verifying and generating all unrefinable partitions of a fixed weight, offering a practical tool for experimentation and data-based conjectures. A complete classification of maximal unrefinable partitions, i.e., those attaining the upper bound for the largest part, was later developed and completed~\cite{aragona2022maximal, aragona2022number}. These works adopt a fully enumerative viewpoint, establishing structural constraints, providing exact counts, and showing bijections with suitable subsets of partitions into distinct parts.

Parallel to these developments, a different but closely related line of research concerns the translation between integer partitions and numerical semigroups. The work of Keith and Nath~\cite{keith2011partitions} introduced what we will refer to in this paper as the KN transformation, which associates a partition to a numerical set and encodes combinatorial information through its hookset. This transformation is central to our study. The fundamental link establishing the bijection between the set of hook lengths of a diagram and the {gaps} of the associated numerical semigroup was formally explored by Constantin et al. \cite{constantin2015numerical}.

The interplay between partitions, numerical sets, and Young diagrams is further explored in~\cite{tutacs2019young}, where the authors provide geometric and tableau-based characterizations for Arf partitions. Constantin, Houston-Edwards, and Kaplan \cite{constantin2015numerical} used this geometric correspondence to derive combinatorial results on structural partition classes, such as {core partitions}. Similar connections are investigated in~\cite{suer2021symmetric}, focusing on symmetric and pseudo-symmetric numerical semigroups via Young diagrams, and in~\cite{burson2023integer}, regarding how semigroup properties influence the associated partition. More recent contributions, such as Ye\c{s}il~\cite{yecsil2025young}, examine refined Young-diagram decompositions for almost-symmetric numerical semigroups.

A separate yet deeply relevant thread concerns the minimal excludant (mex), a statistic that plays a fundamental role in the structure of unrefinable partitions. In unrefinable partitions, the {mex} governs which refinements are forbidden, thereby imposing strong structural constraints~\cite{aragona2023verification}. The combinatorial importance of the {mex} has been emphasized in various contexts: Andrews and Newman~\cite{Andrews2020} and Ballantine and Merca~\cite{Ballantine2020} study its distribution in partitions, while the connection between the {mex} and partitions into distinct parts has been the subject of recent study \cite{kaur2022minimal}. The broader significance of mex-based invariants in combinatorics and game theory is surveyed in Fraenkel and Peled~\cite{fraenkel2015harnessing}. Finally, the study of partition statistics through geometric methods has recently intersected with classical work on partition cranks~\cite{Hopkins2022}.

The initial terms of the sequence enumerating unrefinable partitions are recorded as A179009 in the On-Line Encyclopedia of Integer Sequences (OEIS) \cite{OEIS}.

\subsection{Organization of the paper}
The paper continues as follows: 
in Sec.~\ref{sec:pre}, we recall the necessary background on partitions and numerical sets. We formally define unrefinable and maximal unrefinable partitions, reviewing the known upper bounds for the largest part. Then we present the Keith-Nath transformation and recall the geometric criterion (Proposition~\ref{riconoscimentounref}) that characterizes unrefinability via the hook lengths of the associated Young diagram.

Sec.~\ref{sec:tri} is dedicated to maximal unrefinable partitions of triangular weight $T_n$. We analyze the structural properties of their Young diagrams, highlighting the quasi-symmetric pattern and the role of the extra column. These geometric arguments lead to a combinatorial proof (Theorem~\ref{lambdanraffinabile}) of the bijection between these partitions and partitions of $\frac{n+1}{2}$ into distinct parts through a double embedding construction.

Sec.~\ref{sec:nontri5} and Sec.~\ref{sec:nontri4} extends our analysis to the nontriangular case. We focus on maximal unrefinable partitions with largest part equal to $2n-5$ and $2n-4$, showing that their corresponding Young diagrams exhibit self-conjugate symmetry or a quasi-symmetric structure, respectively. Based on these geometric properties, we obtain the bijective correspondences with suitable partitions into distinct parts (Theorem~\ref{lambdanraffinabile-2n-5} and Theorem~\ref{thm:2n4_unrefinable}).

\section{Preliminaries}\label{sec:pre}
We begin by reviewing the fundamental notions related to  unrefinable partitions and numerical sets and semigroups, along with the specific notation adopted in this work.
\subsection{Unrefinable partitions and related facts}\label{sec:part}

A sequence of positive integers \(\listP\) is called a \emph{partition of \(N\) into distinct parts} if it satisfies
\[
\sum_{i=1}^{t} \l_i = N, \quad \l_1 < \l_2 < \dots < \l_t, \quad t \ge 2.
\]
In this case, we write \(\l \vdash N\) and denote the \emph{length} of the partition by \(|\l| = t\). Let \(\dist_N\) represent the collection of all partitions of \(N\) into distinct parts.

Furthermore, we define the subsets \(\mathbb{D}_{N}^{\,\text{o}}\) and \(\mathbb{D}_{N}^{\,\text{e}}\) of \(\dist_N\) as the partitions in which every part is odd or even, respectively.

For a partition \(\listP \in \dist_N\), we introduce the set of \emph{missing parts}
\[
\mathcal{M}_\l \deq \{1, 2, \dots, \l_t\} \setminus \{\l_1, \l_2, \dots, \l_t\}.
\]
These missing parts can be listed in increasing order as \(\mu_1 < \mu_2 < \dots < \mu_m\), where \(m \ge 0\).

It is straightforward to observe that, if a partition is {refinable}, then its smallest refinable part admits a refinement of the form \(a+b\)~\cite[Proposition~4]{aragona2023verification}. This observation motivates the following definition.

\begin{definition}\label{def_unref}
Let $\listP$ be a partition of $N$ into distinct parts, and let $\mu_1 < \mu_2 < \dots <\mu_m$ denote its missing parts.  
The partition $\l$ is called \emph{refinable} if there exist indices $1 \leq \ell \leq t$ and $1 \leq i < j \leq m$ such that
\[
\mu_i + \mu_j = \l_{\ell},
\]
and it is \emph{unrefinable} otherwise. We denote the set of unrefinable partitions of $N$ by $\mathcal{U}_N$.

An unrefinable partition $\listP \in \mathcal{U}_N$ is said to be \emph{maximal} if 
\[
\l_t = \max_{(\l_1', \l_2', \dots, \l_t') \in \mathcal{U}_N} \l_t'.
\]
The collection of all maximal unrefinable partitions of $N$ is denoted by $\mup{N}$.
\end{definition}

A fundamental aspect in the study of maximal unrefinable partitions is the
behaviour of the largest part~$\lambda_t$.  
For every integer $N$, the value of
$\lambda_t$ is not arbitrary: it is constrained by the combinatorial structure
imposed by unrefinability. 

We recall  that if $T_n =\binom{n+1}{2}$ is the $n$-th triangular number, then every nontriangular integer can be written uniquely, for $1\le d \le n-1$, as
$
T_{n,d} = T_n - d.
$

In the triangular case $N=T_n$ for $n \geq 6$, a maximal unrefinable partition necessarily
satisfies $\lambda_t = 2n-4$,  attaining the bound on the number of
missing parts proved by Aragona et al.~\cite{aragona2022maximal}. For nontriangular numbers, the situation is more delicate but follows similar
principles.  
In a previous work~\cite{aragona2022number}, it was shown that for a maximal unrefinable partition
$\lambda \in \widetilde{\mathcal U}_{T_{n,d}}$, except for the first sporadic values of $n$, the largest part $\lambda_t$
satisfies a sharp upper bound depending on the parity of $n-d$.  More precisely,
if $3< d \le n-1$, then
\[
\begin{cases}
\lambda_t \le 2n-5, & \text{if $n-d$ is even},\\[4pt]
\lambda_t \le 2n-4, & \text{if $n-d$ is odd}.
\end{cases}
\]
The extremal cases $d \in \{1,2,3\}$ yield respectively
\[
\lambda_t \le 2n-2,\qquad
\lambda_t \le 2n-3,\qquad
\lambda_t \le 2n-4.
\]

All these maximal values are summarized in Tab.~\ref{tab:accl-bound}, which collects the sharp upper bounds for every admissible value of $n$ and $d$.

\begin{table}
\centering
\begin{tabular}{c|c|c}
\hline
\textbf{Case} & \textbf{Condition} & \textbf{Upper bound for $\lambda_t$} \\
\hline\hline
Triangular & $N = T_n$ & $\lambda_t = 2n-4$ \\
\hline
\multirow{2}{*}{$3 < d \le n-1$} 
  & $n-d$ odd & $\lambda_t \le 2n-4$ \\
  & $n-d$ even  & $\lambda_t \le 2n-5$ \\
\hline
$d = 3$ & --- & $\lambda_t \le 2n-4$ \\
$d = 2$ & --- & $\lambda_t \le 2n-3$ \\
$d = 1$ & --- & $\lambda_t \le 2n-2$ \\
\hline

\end{tabular}
\bigskip

\caption{Upper bounds for the largest part $\lambda_t$ of a maximal unrefinable 
partition of $T_{n,d}$.}
\label{tab:accl-bound}
\end{table}

Moreover, it is known that an unrefinable
partition cannot contain too many missing parts~\cite{aragona2022maximal}.  
More precisely, if $\lambda=(\lambda_1,\ldots,\lambda_t)$ is an unrefinable
partition with largest part $\lambda_t$, and 
$\mu_1<\cdots<\mu_m$ are its missing parts, then the number of missing 
parts $m$ is bounded above by
\begin{equation}\label{eq:bound-missing}
m \leq \left\lfloor \frac{\lambda_t}{2}\right\rfloor.
\end{equation}
This bound is sharp and is attained by almost all maximal unrefinable partitions, a fact that follows directly from the explicit classification~\cite{aragona2022maximal,aragona2022number}.
For this reason, it is natural to isolate the class of maximal unrefinable
partitions that attain the bound of Eq.~\eqref{eq:bound-missing}.

\begin{definition}
For a fixed integer $N$, we define
\[
\overline{\mathcal{U}}_N \deq 
\left\{\,\lambda = (\l_1,\l_2, \dots, \l_t)\in\mup{N} \;\big|\;
\# \mathcal{M}_\lambda = \left\lfloor \frac{\l_t}{2} \right\rfloor\right\}.
\]
\end{definition}

\begin{remark}\label{rmk:propbarU}
Two useful observations follow from the properties of maximal unrefinable partitions in $\overline{\mathcal{U}}$.
First, no partition in $\overline{\mathcal{U}}$ can contain the value $\frac{\l_t}{2}$, as established in~\cite{aragona2022maximal, aragona2022number}. Second, for any $x \neq \frac{\l_t}{2}$, we have
\[
x \in \lambda \quad \text{if and only if} \quad \lambda_t - x \notin \lambda.
\]
This symmetry property is central to the forthcoming analysis and will be repeatedly exploited.
\end{remark}

The classification provides an explicit description of $\overline{\mathcal{U}}_N$: it consists of all partitions in $\widetilde{\mathcal{U}}_N$ except for certain exceptional ones, which are summarized in Tab.~\ref{tab:exceptions}. It is important to emphasize that, in light of the classification of the triangular case~\cite{aragona2022maximal}, the only nontrivial situation occurs for $T_n$ with $n$ odd. In contrast, when $n$ is even, there exists a single maximal unrefinable partition (namely $\widetilde{\pi}_n$ as listed in Tab.~\ref{tab:exceptions}) and therefore  this case will no longer be examined. The same applies to the extremal cases in terms of $d$ (i.e.\ $1 \leq d \leq 2$), as shown in Tab.~\ref{tab:accl-bound}, which will also not be considered further.

\begin{table}[h!]
\centering
\begin{tabular}{c|c|c}
\hline
\textbf{Case} & \textbf{Condition} & \textbf{Exceptional partitions} \\
\hline\hline
Triangular & $N=T_n$ & $\wtp \deq (1,2,\dots, n-3, n+1, 2n-4)$ \\
\hline
\multirow{2}{*}{Nontriangular} 
  & $N=T_{n,d}$, $n-d$ odd, $d=3$ & $\widetilde\sigma_n \deq (1,2,\dots, n-2, 2n-4)$ \\
  & $N=T_{n,d}$, $n-d$ even, $d=4$ & $\widetilde\tau_n \deq (1,2,\dots, n-2, 2n-5)$ \\
\hline
\end{tabular}
\medskip 

\caption{Exceptional partitions excluded from $\overline{\mathcal{U}}_N$ in the explicit classification.}
\label{tab:exceptions}
\end{table}

\subsection{Numerical sets and Young diagrams}
We now introduce the concept of numerical semigroups, which will play a key role in this study of unrefinable partitions.

\begin{definition}
A \emph{numerical set} $S$ is a subset of the non-negative integers $\mathbb{N}_0$, containing $0$ such that its complement 
\(
S^c = \mathbb{N}_0 \setminus S
\) 
is finite. A numerical set $S$ is called a \emph{numerical semigroup} if it is closed under addition. We denote by $\NSg$ and $\NS$ the collection of all numerical semigroups and sets, respectively.

The elements of $S^c$ are called the \emph{gaps} of $S$, and their number, denoted $G(S)$, is called the \emph{genus} of $S$. The largest element of $S^c$ is the \emph{Frobenius number} $F(S)$. Finally, the \emph{multiplicity} $M(S)$ of $S$ is the smallest positive element of $S$.
\end{definition}

\begin{remark}\label{rmk:nsg=u}
Let $S = \{s_0 = 0,  s_1, s_2, \dots, s_n, \rightarrow\}$ be a numerical semigroup with set of gaps $S^c = \{s_1^c,s_2^c, \dots, s_t^c\}$, where the symbol `$\rightarrow$' means that each integer larger than $s_n$ is also included in the set.
By definition, a numerical semigroup is closed under addition, which implies that no sum of two elements of $S$ can belong to $S^c$. Equivalently, each gap in $S^c$ cannot be expressed as a sum of elements from $S$. This observation shows that the gaps of a numerical semigroup naturally define an unrefinable partition $\lambda$. More specifically, the correspondences between the semigroup and the associated partition are as follows:

\[
\begin{aligned}
& S^c=\{s_1^c,\ldots,s_t^c\} &\rightarrow \quad &\l_S=\l=(\l_1,\ldots,\l_t)\\
    & S=\{0,s_1,\ldots,s_t+1,\rightarrow\} &\rightarrow \quad &\{0\}\cup\mathcal{M}_{\l}\cup\{\l_t+1,\rightarrow\}\\
    & G(S) &\rightarrow \quad &|\l|\\
    & F(S) &\rightarrow \quad &\l_t\\
    &M(S) &\rightarrow \quad & \mu_1\\
 %   & Ap(S,s_1) &\rightarrow \quad &\Vec{p}_{\l}
\end{aligned}
\]
\medskip

In this way, every numerical semigroup can be viewed as encoding an unrefinable partition, with its gaps, genus, Frobenius number and multiplicity corresponding to the parts, length, largest part, minimal excluded element of  $\lambda$, respectively.
Therefore $\NSg \subset \mathcal U$, i.e.\ very numerical semigroup naturally gives rise to an unrefinable partition via its set of gaps.  
However, the converse is not true: not every unrefinable partition corresponds to a numerical semigroup.  
For instance, consider the unrefinable partition 
$
\lambda = (1,2,5,6,8).
$ 
The associated set 
$
S_\lambda = \{0,3,4,7,9,\rightarrow\}
$
fails to be a numerical semigroup, because sums such as $6 = 3+3$ and $8 = 4+4$ are missing from $S_\lambda$. 
This shows that unrefinability of a partition does not guarantee closure under addition in the corresponding set.

\end{remark}

Keith and Nath showed~\cite{keith2011partitions} that every numerical set uniquely defines an integer partition. Indeed, given a numerical set 
	$S=\{0,s_1,\ldots,s_n,\rightarrow\}$
	it is possible to construct a Young diagram $Y_S$ (which in turn induces a partition) by drawing a contiguous  polygonal path starting from the origin in $\mathbb{Z}^2$ as follows: for each integer $0 \leq k \leq F(S)$
	\begin{itemize}
		\item draw an \emph{east} step if $k\in S$,
		\item draw a \emph{north} step if $k\notin S$.
	\end{itemize}
Note that the Young diagram $Y_S$ has $G(S)$ rows and $n$ columns, where $S=\{0,s_1,\ldots,s_n,\rightarrow\}$.
	We refer to the above construction as the \emph{$\KN$ transformation}, mapping a numerical set to its associated Young diagram.
This construction defines a one-to-one correspondence, that we denote by 
\[
\begin{aligned}
\KN: \NS &\rightarrow \{Y : Y \text{ is a Young diagram}\} \\
S &\mapsto Y_S.
\end{aligned}
\]

If $S$ is a numerical set, the Young diagram $Y_S$ obtained via the KN
transformation induces, by proceeding as in Remark~\ref{rmk:nsg=u}, the partition
$\lambda_S = S^{c}$.
When $S$ is a numerical semigroup, the associated partition $\lambda_S$ is
unrefinable.
	
	\begin{example}\label{ex:ns}
		Let $S=\{0,3,6,8,9,11,12,14,\rightarrow\}$ be a numerical set. The associated  Young diagram $Y_S$ is shown in Fig.~\ref{fix:exNS}.
		
		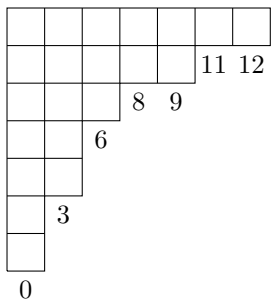
\begin{figure}
		\[
		\begin{tikzpicture}
			\draw (0,0)--(0.5,0)--(0.5,1)--(1,1)--(1,2)--(1.5,2)--(1.5,2.5)--(2.5,2.5)--(2.5,3)--(3.5,3)--(3.5,3.5)--(0,3.5)--(0,0);
			\draw (0.5,3.5)--(0.5,1);
			\draw (1,3.5)--(1,2);
			\draw (1.5,3.5)--(1.5,2.5);
			\draw (2,3.5)--(2,2.5);
			\draw (3,3.5)--(3,3);
			\draw (2.5,3.5)--(2.5,3);
			\draw (0,0.5)--(0.5,0.5);
			\draw (0,1)--(0.5,1);
			\draw (0,1.5)--(1,1.5);
			\draw (0,2)--(1,2);
			\draw (0,2.5)--(1.5,2.5);
			\draw (0,3)--(2.5,3);
			\node at (0.25,-0.25) {$0$};
			\node at (0.75,0.75) {$3$};
			\node at (1.25,1.75) {$6$};
			\node at (1.75,2.25) {$8$};
			\node at (2.25,2.25) {$9$};
			\node at (2.75,2.75) {$11$};
			\node at (3.25,2.75) {$12$};
		\end{tikzpicture}
		\]
		\caption{The Young diagram corresponding to the numerical set of Example~\ref{ex:ns}.}
		\label{fix:exNS}
		\end{figure}
	%	and the corresponding partition via the KN transformation is $(7,5,3,2,2,1,1)$.
	\end{example}

Before proceeding further, we introduce the structural components of a Young
diagram that will be needed in order to analyze its combinatorial properties.
All diagrams are read from left to right and from top to bottom.

\begin{definition}

We denote by $C_i$ the $i$-th column and by $R_j$ the $j$-th row of a Young diagram. Given a cell $c_{i,j}$ in the diagram, the \emph{arm} of $c_{i,j}$ is the number $a(c_{i,j})$ of cells to its right in the $i$-th row, while the \emph{leg} of $c_{i,j}$ is the number $l(c_{i,j})$ is the number of cells below it in the $j$-th column. The \emph{hook} $h_{i,j}$ of the cell $c_{i,j}$ is defined as the sum of its arm and leg, increased by one, and the set of all hooks is called the \emph{hookset}.
\end{definition}

The following lemma gives a general property of the hooks in a Young diagram.
	\begin{lemma}\label{hookinterno}
		Let $Y$ be a Young diagram. If $i,j\neq1$, then
		\begin{equation}\label{equazionehookinterno}
			h_{i,j}=h_{1,j}+h_{i,1}-h_{1,1}.    
		\end{equation}
	\end{lemma}
	\begin{proof}
		We first note that the leg of the cell $c_{1,j}$ is $\#C_j-1$ while its arm is $n-1-j$. Therefore 
		$
		h_{1,j}=\#C_j+n-1-j.
		$
		Similarly, for the cell $c_{i,1}$ we have 
		$
		h_{i,1}=\#R_i+n-2-i.
		$
		Combining the facts above, we obtain
		\[
		h_{1,j}+h_{i,1}-h_{1,1}=\#C_j+n-1-j+\#R_i+n-2-i-(2n-4).
	\]
		Hence, we have
\[
h_{1,j} + h_{i,1} - h_{1,1} = \#R_i - j + \#C_j - i + 1.
\]
Now, $\#R_i - j$ corresponds exactly to the arm of the cell $c_{i,j}$, and $\#C_j - i$ to its leg. It follows that
\[
h_{1,j} + h_{i,1} - h_{1,1} = a(c_{i,j}) + l(c_{i,j}) + 1,
\]
which establishes the claim in Eq.~\eqref{equazionehookinterno}.
	\end{proof}

In \cite{tutacs2019young}, the authors established several properties relating the hookset of a Young diagram to the associated numerical semigroup; we refer to their results for a complete treatment. In the following, we state one of these known results, which will later serve as a template for a version adapted to the set of unrefinable partitions $\mathcal{U}$.

\begin{proposition}[{\cite[Lemma~4]{tutacs2019young}}]\label{riconoscimentoNS}
Let $S = \{0, s_1, \dots, s_n, \rightarrow\}$ be a numerical set with corresponding Young diagram $Y_S$. Then:
\begin{enumerate}
    \item \label{riconoscimentoNS1} the hook length of the cell in the first column and $i$-th row corresponds to the $i$-th largest gap of $S$;
    \item \label{riconoscimentoNS2} for each $0 \le i \le n-1$, the hook length of the top cell in the $i$-th column of $Y_S$ is $F(S) - s_{i-1}$;
    \item \label{riconoscimentoNS3} $S$ is a numerical semigroup if and only if all hook lengths appear in the first column.
\end{enumerate}
\end{proposition}

\begin{example}
		Let $S$ be the numerical set as in Example~\ref{ex:ns}. The hook lengths in the Young diagram are shown below. 
		
		\[
		\begin{tikzpicture}
			\draw (0,0)--(0.5,0)--(0.5,1)--(1,1)--(1,2)--(1.5,2)--(1.5,2.5)--(2.5,2.5)--(2.5,3)--(3.5,3)--(3.5,3.5)--(0,3.5)--(0,0);
			\draw (0.5,3.5)--(0.5,1);
			\draw (1,3.5)--(1,2);
			\draw (1.5,3.5)--(1.5,2.5);
			\draw (2,3.5)--(2,2.5);
			\draw (3,3.5)--(3,3);
			\draw (2.5,3.5)--(2.5,3);
			\draw (0,0.5)--(0.5,0.5);
			\draw (0,1)--(0.5,1);
			\draw (0,1.5)--(1,1.5);
			\draw (0,2)--(1,2);
			\draw (0,2.5)--(1.5,2.5);
			\draw (0,3)--(2.5,3);
			\node at (0.25,0.25) {$1$};
			\node at (0.25,0.75) {$2$};
			\node at (0.25,1.25) {$4$};
			\node at (0.25,1.75) {$5$};
			\node at (0.25,2.25) {$7$};
			\node at (0.25,2.75) {$10$};
			\node at (0.25,3.25) {$13$};
			
			\node at (0.75,1.25) {$1$};
			\node at (0.75,1.75) {$2$};
			\node at (0.75,2.25) {$4$};
			\node at (0.75,2.75) {$7$};
			\node at (0.75,3.25) {$10$};
			
			\node at (1.25,2.25) {$1$};
			\node at (1.25,2.75) {$4$};
			\node at (1.25,3.25) {$7$};
			
			\node at (1.75,2.75) {$2$};
			\node at (1.75,3.25) {$5$};
			\node at (2.25,2.75) {$1$};
			\node at (2.25,3.25) {$4$};
			\node at (2.75,3.25) {$2$};
			\node at (3.25,3.25) {$1$};
		\end{tikzpicture}
		\]
		\medskip
		
\noindent In accordance with the criterion of Proposition~\ref{riconoscimentoNS},
the set $S$ is a numerical semigroup, since all the integers appearing as
hook lengths occur in the first column of $Y_S$.  
Moreover, the first column shows the associated partition
$\lambda_S$, which in this case is given explicitly by
\[
\lambda_S = (1,2,4,5,7,10,13).
\]
	\end{example}

Notice that item~(\ref{riconoscimentoNS3}) of Proposition~\ref{riconoscimentoNS} provides a characterization of numerical semigroups among numerical sets. We now adapt this result to the set of unrefinable partitions $\mathcal{U}$. More precisely, we  derive a criterion to determine whether a partition is unrefinable directly from the Young diagram of the corresponding numerical set obtained via the $\KN$ transformation, thereby strengthening the condition in item~(\ref{riconoscimentoNS3}) to capture the property of unrefinability. The proof is taken from \cite{lorenzo}, but we include it here for completeness. 

\begin{proposition}\label{riconoscimentounref}
		Let $\l=(\l_1,\ldots,\l_t)$ be an unrefinable partition corresponding to the Young tableau $Y_{S_{\l}}$, where $S_{\l}$ is the numerical set associated to $\l$. Then
		\begin{enumerate}
			\item \label{item-ric1} the hook length of the cell in the first column and $i$-th row is $\l_{t-i+1}$;
			\item \label{item-ric2} for each $2\leq i\leq \#\mathcal{M}_{\l}$ the hook length of the top cell of the $i$-th column of $Y_{S_{\l}}$ is equal to $\l_t-\mu_{i-1}$;
			\item \label{item-ric3} $\l$ is an unrefinable partition if and only every length of the hook of the cells of       
			$Y_{S_{\l}}$
			\begin{enumerate}
				\item \label{item-ric3a} is contained in the first column $Y_{S_{\l}}$;\\
				or
				\item \label{item-ric3b} the length of the hook of the cell in the first column and the same row is its double.    
			\end{enumerate}
		\end{enumerate}
	\end{proposition}
	
	\begin{proof}
The proofs of items~(\ref{item-ric1}) and~(\ref{item-ric2}) follow exactly the same arguments as in Proposition~\ref{riconoscimentoNS} givein in~\cite{tutacs2019young}, and they hold in general for all partitions into distinct parts.

Let now $\l$  be an unrefinable partition. If $\l$ corresponds to a numerical semigroup $S_{\l}$, then, by Proposition~\ref{riconoscimentoNS}(\ref{riconoscimentoNS3}), all the cells contain numbers appearing in the first column, so the statement (\ref{item-ric3a}) is proved. If $S_{\l}\notin \NSg$, then there exists $\mu_j\in\mathcal{M}_{\l}$ such that $k\mu_j\in\l$, for some $k > 1$. 
In particular, the unrefinability of $\lambda$ forces $k = 2$. 
Since $2\mu_j \in \lambda$, a row exists by item~(\ref{item-ric1}) whose first cell is labeled $2\mu_j$, and all other cells in that row are of the form $2\mu_j - s_i$, with $s_i \in S_\lambda$ and $s_i < \mu_j$, so in particular there is a cell labeled $2\mu_j - \mu_j = \mu_j$.
Suppose that neither condition~(\ref{item-ric3a}) nor condition~(\ref{item-ric3b}) is satisfied.  
Since condition~(\ref{item-ric3a}) fails, there exists a cell labeled $x$ that does not appear in the first column; by item~(\ref{item-ric1}), this implies $x \notin \lambda$. Let $z$ be the first cell in the same row as $x$, so that $x = z - s_i$ for some $s_i \in S_\lambda$. Negating condition~(\ref{item-ric3b}) gives $z \neq 2x$ and $z = x + s_i$, which leads to a contradiction because $\lambda$ would then be refinable.

Conversely, if only condition~(\ref{item-ric3a}) is satisfied, then $S_\lambda$ is a numerical semigroup, and consequently $\lambda$ is unrefinable.  
If either condition~(\ref{item-ric3a}) or (\ref{item-ric3b}) holds, there exists a cell outside the first column labeled $\mu_j$, whose row begins with $2\mu_j$. By item~(\ref{item-ric1}), we then have $2\mu_j \in \lambda$, which cannot be expressed as $\mu_j + \mu_j$, so the partition remains unrefinable.
	\end{proof} 
	
	\begin{example}
Consider the unrefinable partition $\lambda = (1,2,5,6,8)$.  
If we label each cell of the corresponding Young diagram $Y_{S_\l}$ with its hook length, we obtain

\[
\begin{tikzpicture}
    \draw (0,0)--(0.5,0)--(0.5,1)--(1.5,1)--(1.5,2)--(2,2)--(2,2.5)--(0,2.5)--(0,0);
    \draw (0.5,2.5)--(0.5,1);
    \draw (1,2.5)--(1,1);
    \draw (1.5,2.5)--(1.5,2);
    \draw (0,0.5)--(0.5,0.5);
    \draw (0,1)--(0.5,1);
    \draw (0,1.5)--(1.5,1.5);
    \draw (0,2)--(1.5,2);
    
    \node at (0.25,0.25) {$1$};
    \node at (0.25,0.75) {$2$};
    \node at (0.25,1.25) {$5$};
    \node [color=purple] at (0.25,1.75) {$6$};
    \node [color=NavyBlue] at (0.25,2.25) {$8$};
    
    \node at (0.75,1.25) {$2$};
    \node [color=purple] at (0.75,1.75) {$3$};
    \node at (0.75,2.25) {$5$};
    
    \node at (1.25,1.25) {$1$};
    \node at (1.25,1.75) {$2$};
    \node [color=NavyBlue] at (1.25,2.25) {$4$};
    
    \node at (1.75,2.25) {$1$};
\end{tikzpicture}
\]

\medskip

\noindent
Here, the numbers are colored to emphasize the elements relevant to condition (\ref{item-ric3b}) in Proposition~\ref{riconoscimentounref}. From the diagram, we can see that all the requirements are satisfied, and therefore the partition $\lambda$ is unrefinable.
\end{example}
	
Using the general criterion for unrefinability derived from the Young diagram via the KN transformation, we  recover now the  correspondence between maximal unrefinable partitions of given weight and maximal part and suitable partitions of into distinct parts. In Sec.~\ref{sec:tri}, we  show that the unrefinable partitions in $\widetilde{\mathcal{U}}$ of triangular weight $T_n$ with $n = 2k-1$, which attain the maximal part $\lambda_t = 2n-4$, are in one-to-one correspondence with the partitions of $k$ into distinct parts.
In the following section, as well as in the others, our strategy will be to work systematically with $\overline{\mathcal{U}}$ rather than $\widetilde{\mathcal{U}}$. The reason is that the partitions excluded in passing from $\widetilde{\mathcal{U}}$ to $\overline{\mathcal{U}}$, enumerated in Tab.~\ref{tab:exceptions}, are exceptional and do not follow the general structural pattern.  These exceptional cases will nevertheless be taken into account separately when performing the final counting arguments.
	
\section{Triangular weight}\label{sec:tri}

In this section, we aim to prove the bijection 
\(\mup{T_{n}} \leftrightarrow \mathbb{D}^{\,\text{e}}_{2k} \leftrightarrow \mathbb{D}_k\), where \(n= 2k-1\)  and \(k \ge 4\).  
We establish this correspondence by constructing a double embedding, presented in the two subsections that follow.

\subsection{Proof of \(\mup{T_{n}} \hookrightarrow  \mathbb{D}_k\)}
As anticipated, we focus on the partitions \(\lambda = (\lambda_1, \dots, \lambda_t) \in \overline{\mathcal{U}}_{T_n}\), which are maximal unrefinable partitions with the largest possible number of missing parts.
In this case, \(\overline{\mathcal{U}}_{T_n} = \mup{T_n} \setminus \{\widetilde{\pi}_n\)\}. To achieve our claim, for a given partition \(\lambda\), we consider its associated numerical set \(S_\lambda\) 
and apply the Keith-Nath transformation to construct the corresponding Young diagram $Y_{S_\l}$.\\

Let $\lambda = (\lambda_1, \dots, \lambda_t) \in \overline{\mathcal{U}}_{T_n} = \mup{T_n} \setminus \{\widetilde{\pi}_n\}$ with $n = 2k-1$ and \(k \ge 4\), and let $S_\lambda \in \NS$ be the numerical set such that $S_\lambda^c = \lambda$. We consider the Young diagram
\[
Y_{S_\lambda} = \KN(S_\lambda),
\]
constructed from $S_\lambda$ via the Keith-Nath procedure.  

First, we observe that if $\lambda \in \overline{\mathcal{U}}_{T_n}$, then the largest part satisfies $\lambda_t = 2n-4$, and the number of parts is $t = \lfloor \frac{\lambda_t}{2} \rfloor = n-2$.  We now proceed to establish some structural properties of $Y_{S_\lambda}$. In particular,  we are going to show that the Young diagram $Y_{S_\lambda}$ (with $\l$ as in this section) is \emph{quasi-symmetric}, i.e.\ symmetric except for an additional column that is appended just after the main diagonal.

We start by presenting two results showing the structure of the first row and of  the first column, and the fact that the first row contains a cell labeled $n-2$. 

	\begin{lemma}\label{lemLR}
	Using the above notation, the following claim holds:
	\begin{enumerate}
	\item \label{C1} the first column of $Y_{S_\lambda}$ contains $n-2$ cells, and the hook length of the cell in the $i$-th row of this column is
\[
h_{i,1} = \lambda_{n-2-(i-1)};
\]
		\item \label{R1} the first row of $Y_{S_\lambda}$ contains $n-1$ cells, and the hook length of the cell in the $i$-th column of this row is
\[
h_{1,i} = \lambda_{n-2} - s_{i-1}.
\]
	\end{enumerate}
		
	\end{lemma}
	
	\begin{proof}
	We prove each statement in turn.
	\begin{enumerate}
	\item By Theorem~\ref{riconoscimentounref}(\ref{item-ric1}), the hook length $h_{i,1}$ corresponds to the $i$-th largest element of $S_{\l}^c$. Hence, we obtain $\#C_1=\lvert\l\rvert=n-2$ and $h_{i,1}=\l_{n-2-(i-1)}$.
	\item 	By Theorem~\ref{riconoscimentounref}(\ref{item-ric2}), we have
		\[\#R_1=\#\{s_i\in S_{\l}\mid s_i<\l_{n-2}=2n-4\}=\#\{0\cup\mathcal{M}_{\l}\}.\]
		Therefore, $\#R_1=n-1$ and moreover $h_{1,i}=F(S_{\l})-s_{i-1}=\l_{n-2}-s_{i-1}$.\qedhere
\end{enumerate}
		 	\end{proof}

\begin{lemma}\label{sz}
		There exists an integer $1\leq z\leq n-2$ such that
		$$
		h_{1,z+1}=n-2.
		$$
	\end{lemma}
	
	\begin{proof}
	Since $\l\in\overline{\mathcal{U}}_{T_n}$, we have $n-2\notin\l=S_{\l}^c$. Hence, there exists an element $s_z\in S_{\l}$ with $s_z=n-2$. It then follows that
		\[
		h_{1,z+1}=\l_{n-2}-s_z=n-2.\qedhere
		\]
\end{proof}

The following lemma  highlights an almost perfect symmetry between the first row and the first column of the Young diagram $Y_{S_\lambda}$. Specifically, for smaller indices $0 \le i \le z-1$, the hook length of the cell in the $(i+1)$-th row of the first column coincides with that of the cell in the first row and $(i+1)$-th column. For larger indices $z+1 \le i \le n-2$, the hook length of the cell in the first row and $(i+1)$-th column matches that of the cell in the first column and $i$-th row. In other words, the main part of the first row and first column is mirrored, except for a slight shift corresponding to the cell associated with $z$. 

\begin{lemma}\label{simmetriaC1R1}
		Using the above notation and letting \(z\) be the integer as in Lemma~\ref{sz}, we have that if $0\leq i\leq z-1$, then
		\[
		h_{i+1,1}=h_{1,i+1},
		\]
		whereas if $z+1\leq i\leq n-2$, then
		\[
		h_{1,i+1}=h_{i,1}.
		\]
	\end{lemma}
	\begin{proof}
First, suppose that $i \le z-1$. In this case, we consider the elements $s_i \in S_\lambda$ with $s_i < s_z$, so that $s_i \in \{0\} \cup \mathcal{M}_\lambda$. Since $\lambda \in \overline{\mathcal{U}}_{T_n}$, for any $x \neq \frac{\lambda_t}{2}$, we have $x \in \lambda$ if and only if $\lambda_t - x \in \mathcal{M}_\lambda$ (cf.\ also Remark~\ref{rmk:propbarU}). It follows that 
\[
\lambda_{n-2} - s_i = \lambda_{n-2-i},
\] 
and thus
\[
h_{i+1,1} = \lambda_{n-2-i} = \lambda_{n-2} - s_i = h_{1,i+1}.
\]

Next, suppose that $i \ge z+1$. In this case, $s_i > s_z$, and we similarly obtain
\[
\lambda_{n-2} - s_i = \lambda_{n-2-(i-1)},
\] 
which implies
\[
h_{i,1} = \lambda_{n-2-(i-1)} = \lambda_{n-2} - s_i = h_{1,i+1}.\qedhere
\]
	\end{proof}
	
	Fig.~\ref{fig:shape1} provides a schematic illustration of the symmetry between the first row and the first column in the Young diagram of a maximal unrefinable partition in $\overline{\mathcal{U}}_{T_n}$, highlighting the special cell corresponding to the index $z$. This diagram visually reflects the properties we have established so far.
	
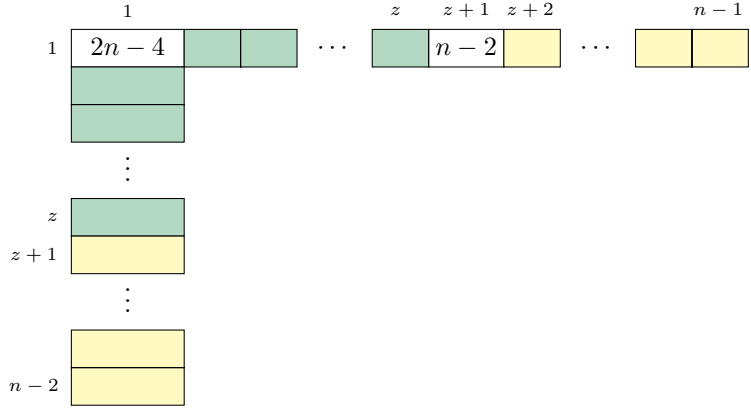
\begin{figure}	
	\begin{tikzpicture}
		\draw (1,0) rectangle (-0.5,-0.5);
		\draw[color=black, fill=ForestGreen!30] (-0.5,-0.5) rectangle (1,-1) rectangle (-0.5,-1.5);
		\draw[color=black, fill=ForestGreen!30] (-0.5,-2.25) rectangle (1,-2.75);
		\draw[color=black, fill=yellow!30] (1,-2.75) rectangle (-0.5,-3.25);
		\draw[color=black, fill=yellow!30] (-0.5,-4) rectangle (1,-4.5) rectangle (-0.5,-5);
		\draw[color=black, fill=ForestGreen!30] (1,0) rectangle (1.75,-0.5) rectangle (2.5,0);
		\draw[color=black, fill=ForestGreen!30] (3.5,0) rectangle (4.25,-0.5);
		\draw (4.25,-0.5) rectangle (5.25,0);
		\draw[color=black, fill=yellow!30] (5.25,0)rectangle (6,-0.5);
		\draw[color=black, fill=yellow!30] (7,0) rectangle (7.75,-0.5) rectangle (8.5,0); 
		\node at (0.25,-1.75) {$\vdots$};
		\node at (0.25,-3.5) {$\vdots$};
		\node at (3,-0.25) {$\cdots$};
		\node at (4.75,-0.25) {$n-2$};
		\node at (6.5,-0.25) {$\cdots$};
		\node at (0.25,-0.25) {$2n-4$};
		\node at (0.25,0.25) {\scriptsize{\textbf{$1$}}};
		\node at (-0.75,-0.25) {\scriptsize{\textbf{$1$}}};
		\node at (3.81,0.25) {\scriptsize{\textbf{$z$}}};
		\node at (4.75,0.25) {\scriptsize{\textbf{$z+1$}}};
		\node at (5.6,0.25) {\scriptsize{\textbf{$z+2$}}};
		\node at (8.1,0.25) {\scriptsize{\textbf{$n-1$}}};
		\node at (-0.75,-2.5) {\scriptsize{\textbf{$z$}}};
		\node at (-1,-3) {\scriptsize{\textbf{$z+1$}}};
		\node at (-1,-4.75) {\scriptsize{\textbf{$n-2$}}};
	\end{tikzpicture}
	\caption{Symmetry between first row and first column in $Y_{S_\lambda}$, highlighting the special cell $z$.}
		\label{fig:shape1}	
\end{figure}

The following result is an application of Lemma~\ref{hookinterno} to the specific case of Young diagrams arising from maximal unrefinable partitions in $\overline{\mathcal{U}}_{T_n}$.

\begin{corollary}
		Let $\l\in \overline{\mathcal{U}}_{T_n}$ and let $Y_{S_{\l}}$ be the associated Young diagram. If $2\leq i\leq z$,
		\begin{equation}\label{equazionehookdq}
			h_{i,j}=\l_{n-2}-s_{i-1}-s_{j-1},
		\end{equation}
		whereas if $z+1\leq i\leq n-2$,
		\begin{equation}\label{equazionehookfq}
			h_{i,j}=\l_{n-2}-s_i-s_{j-1}.
		\end{equation}
	\end{corollary}
	
	\begin{proof}
		Suppose $2\leq i\leq z$. By Eq.~\eqref{equazionehookinterno} and Lemma~\ref{simmetriaC1R1}, we have
		\[
		h_{i,j}=h_{1,j}+h_{i,1}-h_{1,1}=\l_{n-2}-s_{j-1}+\l_{n-2}-s_{i-1}-\l_{n-2},
		\]
		which gives Eq.~\eqref{equazionehookdq}.
		
		Now suppose $z+1\leq i\leq n-2$. By Lemma~\ref{simmetriaC1R1}, we know that $h_{i,1}=h_{1,i+1}=\l_{n-2}-s_i$. Substituting this into Eq.~\eqref{equazionehookinterno} yields Eq.~\eqref{equazionehookfq}.
	\end{proof}
	
We now show that the diagram contains a principal square starting from the top-left corner, and we are interested in understanding its structure and the column immediately following it. In particular, a potential cell in position $(z+1,z+1)$ would have a negative hook and therefore would not exist. Consequently, the column immediately after this square contains exactly $z$ cells, and the main diagonal itself consists precisely of $z$ cells. These observations are formalized in the following lemma.

\begin{lemma}\label{lemma:main_diagonal}
Consider the Young diagram $Y_{S_\lambda}$ associated to a maximal unrefinable partition $\lambda \in \overline{\mathcal{U}}_{T_n}$. Then the following properties hold:

\begin{enumerate}
    \item \label{lemma:main_diagonal1} the cell $c_{z+1,z+1}$ does not exist in the diagram;
    \item \label{lemma:main_diagonal2} the column $C_{z+1}$ contains exactly $z$ cells;
    \item \label{lemma:main_diagonal3} the main diagonal of $Y_{S_\lambda}$ contains exactly $z$ cells; equivalently,
    \[
    h_{z,z} > 0 \quad\text{and}\quad h_{z+1,z+1} < 0.
    \]
\end{enumerate}
\end{lemma}

\begin{proof}
We prove each statement in turn.
\begin{enumerate}
    \item The hook of the cell $c_{z+1,z+1}$ would be
    $$
    h_{z+1,z+1} = \lambda_{n-2} - s_{z+1} - s_{z+2} = n-2 - s_{z+2} < 0,
    $$
    since $s_{z+2} > s_{z+1}$. Therefore, the cell $c_{z+1,z+1}$ cannot belong to the diagram.
    
    \item From the previous point, $h_{z+1,z+1}<0$, while $h_{z,z+1} = n-2 - s_z > 0$. It follows that the column $C_{z+1}$ contains exactly $z$ cells.
    
    \item By (\ref{lemma:main_diagonal1}), $h_{z+1,z+1} < 0$. On the other hand, $h_{z,z} > 0$, since $h_{z,z+1} > 0$ by (\ref{lemma:main_diagonal2}). Hence the main diagonal of $Y_{S_\lambda}$ contains exactly $z$ cells.
\end{enumerate}
\end{proof}

Fig.~\ref{fig:shape2} provides a graphical illustration of the previous lemma, highlighting the principal square in the Young diagram and the column immediately following it, showing which cells belong to the diagram and which do not.  \\

\begin{figure}	
	\begin{tikzpicture}
		\draw (1,0) rectangle (-0.5,-0.5);
		\draw[color=black, fill=ForestGreen!30] (-0.5,-0.5) rectangle (1,-1) rectangle (-0.5,-1.5);
		\draw[color=black, fill=ForestGreen!30] (-0.5,-2.25) rectangle (1,-2.75);
		\draw[color=black, fill=yellow!30] (1,-2.75) rectangle (-0.5,-3.25);
		\draw[color=black, fill=yellow!30] (-0.5,-4) rectangle (1,-4.5) rectangle (-0.5,-5);
		\draw[color=black, fill=ForestGreen!30] (1,0) rectangle (1.75,-0.5) rectangle (2.5,0);
		\draw[color=black, fill=ForestGreen!30] (3.5,0) rectangle (4.25,-0.5);
		\draw (4.25,-0.5) rectangle (5.25,0);
		\draw[color=black, fill=yellow!30] (5.25,0)rectangle (6,-0.5);
		\draw[color=black, fill=yellow!30] (7,0) rectangle (7.75,-0.5) rectangle (8.5,0); 
		\node at (0.25,-1.75) {$\vdots$};
		\node at (0.25,-3.5) {$\vdots$};
		\node at (3,-0.25) {$\cdots$};
		\node at (4.75,-0.25) {$n-2$};
		\node at (6.5,-0.25) {$\cdots$};
		\node at (0.25,-0.25) {$2n-4$};
		\node at (0.25,0.25) {\scriptsize{\textbf{$1$}}};
		\node at (-0.75,-0.25) {\scriptsize{\textbf{$1$}}};
		\node at (3.81,0.25) {\scriptsize{\textbf{$z$}}};
		\node at (4.75,0.25) {\scriptsize{\textbf{$z+1$}}};
		\node at (5.6,0.25) {\scriptsize{\textbf{$z+2$}}};
		\node at (8.1,0.25) {\scriptsize{\textbf{$n-1$}}};
		\node at (-0.75,-2.5) {\scriptsize{\textbf{$z$}}};
		\node at (-1,-3) {\scriptsize{\textbf{$z+1$}}};
		\node at (-1,-4.75) {\scriptsize{\textbf{$n-2$}}};
		\node at (4.75,-1.5) {$\vdots$};
		\node at (2.5,-2.5) {$\cdots$};
		\node at (2.5,-1.5) {$\ddots$};
		\node at (0.25,-1.75) {$\vdots$};
		\node at (0.25,-3.5) {$\vdots$};
		\node at (3,-0.25) {$\cdots$};
		\node at (6.5,-0.25) {$\cdots$};
		\draw (3.5,-2.25) rectangle (4.25,-2.75);
		\draw (4.25,-2.25) rectangle (5.25,-2.75);
\draw[pattern=north east lines, pattern color=black!55, dashed] 
      (4.25,-2.75) rectangle (5.25,-3.25);

	\end{tikzpicture}
	\caption{Principal square and first column beyond it in $Y_{S_\lambda}$, illustrating the existence of $z$ cells in the diagonal and in column $C_{z+1}$.}
		\label{fig:shape2}	
\end{figure}

The following result formalizes the row-column behaviour one expects from a symmetric Young diagram, except for the single extra column appearing immediately after the main diagonal. This asymmetry is precisely what causes the change of regime at $i=z+1$.

\begin{lemma}\label{lem:simquad}
		Let $1\leq i\leq z$. Then
		$
		\#R_i=\#C_i+1.
		$
		Otherwise, if $z+1\leq i\leq n-2$, we have
		$
		\#R_i=\#C_{i+1}.
		$
	\end{lemma}
	\begin{proof}
		Fix $1\leq i\leq z$. From
		\begin{equation}\label{RCQ}
			\#R_i+n-2-i=h_{i,1}=h_{1,i}=\#C_i+n-1-i
		\end{equation}
		it follows immediately that $\#R_i=\#C_i+1$.
		
		Now let $z+1\leq i\leq n-2$. In this case $h_{i,1}=h_{1,i+1}$, hence 
		\begin{equation}\label{RCFQ}
			\#R_i+n-2-i=\#C_{i+1}+n-1-i+1,
		\end{equation}
		which gives $\#R_i=\#C_{i+1}$.
	\end{proof}

We now prove the hook-symmetry properties of \(Y_{S_\lambda}\). 
Inside the \(z\times z\) square determined by the main diagonal, the hooks behave exactly as in a symmetric Young diagram. 
Beyond this square, the presence of the extra column creates a shifted symmetry, pairing hooks in row \(i\) with those in column \(i+1\). 
The following result formalizes these two regimes.

\begin{lemma}\label{hooksymmetryglobal}
The hook structure of \(Y_{S_\lambda}\) satisfies:
\begin{enumerate}
    \item if \(1\le i,j \le z\), then
    \[
        h_{i,j}=h_{j,i};
    \]
    \item if \(z+1 \le i \le n-2\), then for every \(1 \leq j \leq z\),
    \[
        h_{i,j}=h_{j,i+1}.
    \]
\end{enumerate}
\end{lemma}

\begin{proof}
We prove each statement in turn.
\begin{enumerate}
    \item For \(1\le i,j\le z\), by Lemma~\ref{lem:simquad} we have
    \[
    \begin{aligned}
        h_{i,j}
        &=\#R_i-(j-1)+\#C_j-(i-1)-1\\
        &=\#C_i-(j-1)+\#R_j-(i-1)-1\\
        &=h_{j,i}.
    \end{aligned}
    \]

    \item For \(z+1\le i\le n-2\), using Eq.~\eqref{equazionehookfq} and Eq.~\eqref{equazionehookdq}, we have
    \[
        h_{i,j}
        =\lambda_{n-2}-s_i-s_{j-1}
        =\lambda_{n-2}-s_{j-1}-s_i
        =h_{j,i+1}.\qedhere
    \]
\end{enumerate}
\end{proof}

To conclude this list of structural properties, we relate the hooks lying on the main diagonal to those touching the extra column.

\begin{lemma}\label{d2z+1}
Let $1 \le i \le z$. Then the hook length $h_{i,i}$ of the cell on the main diagonal in row $i$ is exactly twice the hook length of the cell in the same row $i$ and in the extra column $C_{z+1}$:
\[
h_{i,i} = 2\,h_{i,z+1}.
\]
\end{lemma}
\begin{proof}
We compute directly:
\[
\begin{aligned}
h_{i,i} &= \lambda_{n-2} - 2 s_{i-1} \\
        &= 2 (n-2 - s_{i-1}) \\
        &= 2 (\lambda_{n-2} - s_z - s_{i-1}) \\
        &= 2 h_{i,z+1}.
\end{aligned}
\]
\end{proof}

We have shown that  the Young diagrams corresponding to maximal unrefinable partitions of triangular numbers exhibit a staircase-like structure with a main diagonal and an extra column beyond the diagonal. 
Now we can count the number of cells of $Y_{S_{\l}}$.

\begin{proposition}
		Let  $\l\in \overline{\mathcal{U}}_{T_{n}}$. Then the number of cells of $Y_{S_{\l}}$ is
		$3n-3.$
		Therefore, the sum of the hook lengths on the main diagonal is equal to $3n-3$.
	\end{proposition}
	\begin{proof}
		The total number of cells in the diagram is equal to the sum of the number of cells in each row:
		\[
		\begin{aligned}
			\#Y_{S_{\l}}&=\sum_{i=1}^{n-2}\#R_i\\
			&=\sum_{i=1}^{n-2} h_{i,1}-(\#C_1-i)\\
			&=\sum_{i=1}^{n-2} h_{i,1}-\sum_{i=1}^{n-2}\#C_1+\sum_{i=1}^{n-2}i\\
			&=T_n-(n-2)(n-2)+T_{n-2}\\
			&=\frac{n(n+1)}{2}-(n-2)^2+\frac{(n-2)(n-1)}{2}\\
			&=\frac{1}{2}(6n-6)\\
			&=3n-3.
		\end{aligned}
		\]
	\end{proof}

The following proposition highlights a key property of Young diagrams corresponding to maximal unrefinable partitions in the triangular case $T_n$, with $n=2k-1$. It shows how the hook lengths along the extra column beyond the main diagonal naturally encode a sum that equals $k$, providing a direct combinatorial link between these partitions and the partitions of $k$ into distinct parts. This observation will serve as a cornerstone for establishing the desired bijection.

\begin{proposition}\label{etainDistinct}
Let $\lambda \in \overline{\mathcal{U}}_{T_n}$, where $n=2k-1$. Then 
\[
\sum_{i=2}^{z} h_{i,z+1} = k.
\]
\end{proposition}

\begin{proof}
We first compute the sum of the hooks along the main diagonal starting from the second row:
\[
\sum_{i=2}^{z} h_{i,i} = 3n - 3 - (2n - 4) = n + 1 = 2k.
\]
By Lemma~\ref{d2z+1}, each diagonal hook is twice the corresponding hook in the extra column:
\[
h_{i,i} = 2 h_{i,z+1}.
\]
Hence, dividing both sides by $2$, we obtain
$
\sum_{i=2}^{z} h_{i,z+1} = k,
$
as claimed.
\end{proof}

\begin{remark}
Proposition~\ref{etainDistinct} can be interpreted as providing an embedding of the set $\overline{\mathcal{U}}_{T_{2k-1}}$ into the set of partitions $\mathbb{D}_{k}$. 
Observe that the same property that allows a correspondence with partitions of $k$ into distinct parts can also be interpreted in terms of partitions of $2k$ into even (distinct) parts. 
In this perspective, one can read the contribution of each part directly along the main diagonal of the Young diagram, without referring to the extra column. The following example illustrates these features concretely.
\end{remark}

\begin{example}\label{exa:main}
Consider the partition 
\[
\lambda=(1,2,3,4,5,6,7,8,11,14,16,17,26)\in \overline{\mathcal{U}}_{T_{15}}, 
\]
with $k$  equal to $8=\frac{15+1}{2}$. We display in Fig.~\ref{fig:ex-final} its Young diagram and the corresponding symmetries. Looking at the main diagonal, or at the extra column, one find the corresponding partition into distinct parts $(1,3,4) \in \mathbb D_8$.

\begin{figure}
\[
\begin{tikzpicture}
			\draw (0,0)--(0.5,0)--(0.5,4)--(1.5,4)--(1.5,4.5)--(2.5,4.5)--(2.5,5)--(3,5)--(3,6)--(7,6)--(7,6.5)--(0,6.5)--(0,0);
			
			\draw (0,0.5)--(0.5,0.5);
			\draw (0,1)--(0.5,1);
			\draw (0,1.5)--(0.5,1.5);
			\draw (0,2)--(0.5,2);
			\draw (0,2.5)--(0.5,2.5);
			\draw (0,3)--(0.5,3);
			\draw (0,3.5)--(0.5,3.5);
			\draw (0,4)--(0.5,4);
			\draw (0,4.5)--(1.5,4.5);
			\draw (0,5)--(2.5,5);
			\draw (0,5.5)--(3,5.5);
			\draw (0,6)--(3.5,6);
			
			\draw (0.5,6.5)--(0.5,4);
			\draw (1,6.5)--(1,4);
			\draw (1.5,6.5)--(1.5,4);
			\draw (2,6.5)--(2,4.5);
			\draw (2.5,6.5)--(2.5,4.5);
			\draw (3,6.5)--(3,6);
			\draw (3.5,6.5)--(3.5,6);
			\draw (4,6.5)--(4,6);
			\draw (4.5,6.5)--(4.5,6);
			\draw (5,6.5)--(5,6);
			\draw (5.5,6.5)--(5.5,6);
			\draw (6,6.5)--(6,6);
			\draw (6.5,6.5)--(6.5,6);

			\draw[color=black, fill=yellow!30] (0,0) rectangle (0.5,0.5);
			\draw[color=black, fill=yellow!30] (0,0.5) rectangle (0.5,1);
			\draw[color=black, fill=yellow!30] (0,1) rectangle (0.5,1.5);
			\draw[color=black, fill=yellow!30] (0,1.5) rectangle (0.5,2);
			\draw[color=black, fill=yellow!30] (0,2) rectangle (0.5,2.5);
			\draw[color=black, fill=yellow!30] (0,2.5) rectangle (0.5,3);
			\draw[color=black, fill=yellow!30] (0,3) rectangle (0.5,3.5);
			\draw[color=black, fill=yellow!30] (0,3.5) rectangle (0.5,4);
			\draw[color=black, fill=yellow!30] (0,4) rectangle (0.5,4.5);
			\draw[color=black, fill=yellow!30] (0.5,4) rectangle (1,4.5);
			\draw[color=black, fill=yellow!30] (1,4) rectangle (1.5,4.5);
			
			\draw[color=black, fill=yellow!30] (7,6.5) rectangle (6.5,6);
			\draw[color=black, fill=yellow!30] (6.5,6.5) rectangle (6,6);
			\draw[color=black, fill=yellow!30] (6,6.5) rectangle (5.5,6);
			\draw[color=black, fill=yellow!30] (5.5,6.5) rectangle (5,6);
			\draw[color=black, fill=yellow!30] (5,6.5) rectangle (4.5,6);
			\draw[color=black, fill=yellow!30] (4.5,6.5) rectangle (4,6);
			\draw[color=black, fill=yellow!30] (4,6.5) rectangle (3.5,6);
			\draw[color=black, fill=yellow!30] (3.5,6.5) rectangle (3,6);
			\draw[color=black, fill=yellow!30] (3,6.5) rectangle (2.5,6);
			\draw[color=black, fill=yellow!30] (3,6) rectangle (2.5,5.5);
			\draw[color=black, fill=yellow!30] (3,5.5) rectangle (2.5,5);
			
			\draw[color=black, fill=red!30] (0,6.5) rectangle (0.5,6) rectangle (1,5.5) rectangle (1.5,5) rectangle (2,4.5);
			
			\draw[color=black, fill=ForestGreen!30] (0,6) rectangle (0.5,5.5) rectangle (0,5) rectangle (0.5,4.5) rectangle (1,5) rectangle (1.5,4.5);
			\draw[color=black, fill=ForestGreen!30] (0.5,5.5) rectangle (1,5);
			\draw[color=black, fill=ForestGreen!30] (0.5,6.5) rectangle (1,6) rectangle (1.5,6.5) rectangle (2,6) rectangle (1.5,5.5) rectangle (2,5);
			\draw[color=black, fill=ForestGreen!30] (1,6) rectangle (1.5,5.5);
			
			\draw[color=black, fill=pink!30] (2,6.5) rectangle (2.5,6);
		   \draw[color=purple, fill=pink!30, line width=1.2pt] (2.5,6) rectangle (2,5.5) rectangle (2.5,5) rectangle (2,4.5);

			\node at (0.25,0.25) {$1$};
			\node at (0.25,0.75) {$2$};
			\node at (0.25,1.25) {$3$};
			\node at (0.25,1.75) {$4$};
			\node at (0.25,2.25) {$5$};
			\node at (0.25,2.75) {$6$};
			\node at (0.25,3.25) {$7$};
			\node at (0.25,3.75) {$8$};
			\node at (0.25,4.25) {$11$};
			\node at (0.25,4.75) {$14$};
			\node at (0.25,5.25) {$16$};
			\node at (0.25,5.75) {$17$};
			\node at (0.25,6.25) {$26$};
			
			\node at (0.75,4.25) {$2$};
			\node at (0.75,4.75) {$5$};
			\node at (0.75,5.25) {$7$};
			\node at (0.75,5.75) {$8$};
			\node at (0.75,6.25) {$17$};
			
			\node at (1.25,4.25) {$1$};
			\node at (1.25,4.75) {$4$};
			\node at (1.25,5.25) {$6$};
			\node at (1.25,5.75) {$7$};
			\node at (1.25,6.25) {$16$};
			
			\node at (1.75,4.75) {$2$};
			\node at (1.75,5.25) {$4$};
			\node at (1.75,5.75) {$5$};
			\node at (1.75,6.25) {$14$};
			
			\node at (2.25,4.75) {$1$};
			\node at (2.25,5.25) {$3$};
			\node at (2.25,5.75) {$4$};
			\node at (2.25,6.25) {$13$};
			
			\node at (2.75,5.25) {$1$};
			\node at (2.75,5.75) {$2$};
			\node at (2.75,6.25) {$11$};
			
			\node at (3.25,6.25) {$8$};
			\node at (3.75,6.25) {$7$};
			\node at (4.25,6.25) {$6$};
			\node at (4.75,6.25) {$5$};
			\node at (5.25,6.25) {$4$};
			\node at (5.75,6.25) {$3$};
			\node at (6.25,6.25) {$2$};
			\node at (6.75,6.25) {$1$};

		\end{tikzpicture}
	\]
	
	\caption{Young diagram of a maximal unrefinable partition illustrating the main diagonal, the extra column and its symmetries.}
	\label{fig:ex-final}
	\end{figure}
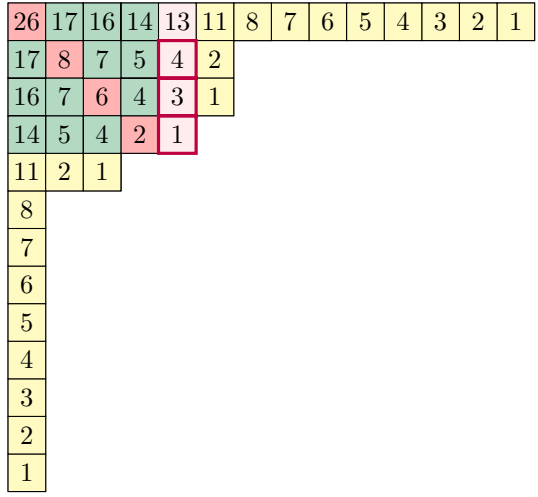
\end{example}

In the following section, we will complete the proof of the bijection, i.e.\ we show how each element of $\overline{\mathcal{U}}_{T_n}$ can be uniquely associated with a distinct-part partition of $k$, except for the partition $(3,k-3)$. Since $\#\mup{T_n} = \#\overline{\mathcal{U}}_{T_n} + 1$, this discrepancy accounts precisely for the missing case, giving us the desired correspondence.

\subsection{Proof of \(\mathbb{D}_k \hookrightarrow \mup{T_{n}} \)}\label{sec:ritornotriang}
Let $\eta=(\eta_1,\cdots,\eta_l)\in\mathbb{D}_k\setminus\{(3,k-3)\}$, and let us define $\eta^*=(\eta_1,\cdots,\eta_l,n-2)$ and $2\eta^*=(2\eta_1,\cdots,2\eta_l,2n-4)$, where $n=2k-1$.

We construct the Young diagram $Y_{2\eta^*}$ by imposing the following conditions:
\begin{align}
h_{1,1} &= 2n - 4, \nonumber \\
h_{i,i} &= 2\eta_{\,l-(i-2)}, \qquad 2 \le i \le l+1, \nonumber \\
a(c_{i,i}) &= l(c_{i,i}) + 1, \qquad 1 \le i \le l+1. \label{equazionecostruzione'}
\end{align}

From the Young diagram $Y_{2\eta^*}$ constructed above, in the following we will apply the inverse of the KN transformation, obtaining a numerical set 
$
S_{2\eta^*} = \KN^{-1}(Y_{2\eta^*})
$.
We will define the associated partition ${\lambda} = S_{2\eta^*}^c$, which we will show to be unrefinable.

\begin{example}
		We consider the partition $\eta=(1,3,4)\in\mathbb{D}_8$, hence $2\eta^*=(2,6,8,26)$. The resulting Young diagram $Y_{2\eta^*}$ obtained with the rules as above  is displayed in Fig.~\ref{fig:Y2eta-example}
		\begin{figure}
		\[
		\begin{tikzpicture}
			\draw[color=black, fill=ForestGreen!30] (0,0) rectangle (0.5,-0.5) rectangle (0,-1) rectangle (0.5,-1.5) rectangle (0,-2) rectangle (0.5,-2.5) rectangle (0,-3) rectangle (0.5,-3.5) rectangle (0,-4) rectangle (0.5,-4.5) rectangle (0,-5) rectangle (0.5,-5.5) rectangle (0,-6) rectangle (0.5,-6.5);
			\draw[color=black, fill=ForestGreen!30] (0.5,0) rectangle (1,-0.5) rectangle (1.5,0) rectangle (2,-0.5) rectangle (2.5,0) rectangle (3,-0.5) rectangle (3.5,0) rectangle (4,-0.5) rectangle (4.5,0) rectangle (5,-0.5) rectangle (5.5,0) rectangle (6,-0.5) rectangle (6.5,0) rectangle (7,-0.5);
			\node at (0.25,-0.25) {$26$};
			\draw[color=black, fill=yellow!30] (0.5,-2.5) rectangle (1,-2) rectangle (0.5,-1.5) rectangle (1,-1) rectangle (0.5,-0.5);
			\draw[color=black, fill=yellow!30] (1,-0.5) rectangle (1.5,-1) rectangle (2,-0.5) rectangle (2.5,-1) rectangle (3,-0.5);
			\node at (0.75,-0.75) {$8$};
			\draw[color=black, fill=pink!30] (1,-2.5) rectangle (1.5,-2) rectangle (1,-1.5) rectangle (1.5,-1) rectangle (2,-1.5) rectangle (2.5,-1) rectangle (3,-1.5);
			\node at (1.25,-1.25) {$6$};
			\draw[color=black, fill=red!30] (1.5,-1.5) rectangle (2,-2) rectangle (2.5,-1.5);
			\node at (1.75,-1.75) {$2$};
		%	\node at (0.25,0.25) {\scriptsize{$1$}};
		%	\node at (-0.25,-0.25) {\scriptsize{$1$}};
		%	\node at (1.75,0.25) {\scriptsize{$4$}};
		%	\node at (-0.25,-1.75) {\scriptsize{$4$}};
		\end{tikzpicture}
		\]
		\caption{Young diagram $Y_{2\eta^*}$ for $\eta=(1,3,4)$.}
		\label{fig:Y2eta-example}
		\end{figure}
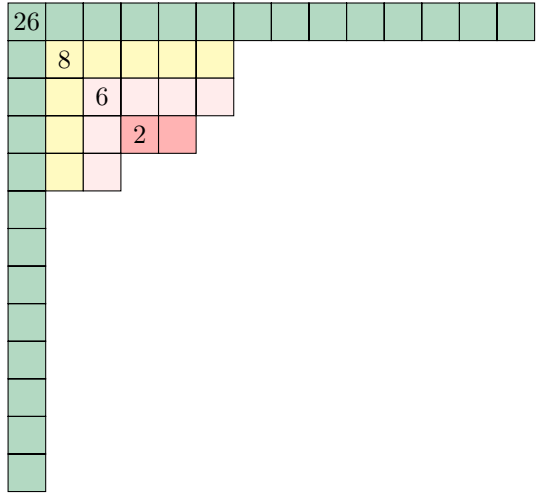
	\end{example}

As in the previous section, we will now derive several structural properties of the Young diagram $Y_{2\eta^*}$, beginning with results concerning the square determined by the main diagonal.

\begin{lemma}\label{rowsColsY}
Let $Y_{2\eta^*}$ be the Young diagram constructed above. Then:
\begin{enumerate}
    \item \label{C_i'} $\#C_1 = n-2$, {and for} $2\leq i\leq l+1$ we have $\#C_i=\eta_{l-(i-2)}+(i-1)$;
   
    \item \label{R_i'} $\#R_1 = n-1$, {and for} $2\leq i\leq l+1$ we have $\#C_i=\eta_{l-(i-2)}+i$.
    \end{enumerate}
\end{lemma}

\begin{proof}
We prove each statement in turn.
\begin{enumerate}
\item From Eq.~\eqref{equazionecostruzione'} we have
\[
h_{1,1}=1+a(c_{1,1})+l(c_{1,1})=2+2l(c_{1,1}).
\]
We know that $h_{1,1}=2n-4$, then we obtain
\[
\#C_1=l(c_{1,1})+1=n-2.
\]
For $2\leq i\leq l+1$ we have $h_{i,i}=2\eta_{l-(i-2)}$, then 
\[
\#C_i=i+l(c_{i,i})=i+\frac{2\eta_{l-(i-2)}-2}{2}=\eta_{l-(i-2)}+(i-1).
\]
\item We have
\[
\#R_1=1+a(c_{1,1})=1+\frac{2n-4}{2}=n-1.
\]
For $2\leq i\leq l+1$ we have $h_{i,i}=2\eta_{l-(i-2)}$, we obtain
\[
\#R_i=i+a(c_{i,i})=i+\frac{h_{i,i}}{2}=\eta_{l-(i-2)}+i.\qedhere
\]
\end{enumerate} 
 \end{proof}

\begin{lemma}\label{hookSymmetrySquare}
Let us consider the square in $Y_{2\eta^*}$ defined by the first $l+1$ rows and columns along the main diagonal. Then:
\begin{enumerate}
    \item \label{C_1R_1q} for $2 \le i \le l+1$, the hooks in the first row and column of this square satisfy
    \[
    h_{1,i} = h_{i,1};
    \]
    \item for $2 \le i,j \le l+1$, the hooks within the square are symmetric across the main diagonal, i.e.\
    \[
    h_{i,j} = h_{j,i}.
    \]
\end{enumerate}
\end{lemma}

\begin{proof}
We prove the two points separately.
\begin{enumerate}
    \item Using Lemma~\ref{rowsColsY}, we have
    \[
    \begin{aligned}
    h_{i,1} &= \#R_i + l(c_{i,1}) \\
    &= \eta_{l-(i-2)} + i + (n-2-i) \\
    &= \eta_{l-(i-2)} + n-2 \\
    &= \eta_{l-(i-2)} + (i-1) + (n-1-i) \\
    &= \#C_i + a(c_{1,i}) \\
    &= h_{1,i}.
    \end{aligned}
    \]
    \item By Lemma~\ref{rowsColsY} and using the previous point, we obtain
    \[
    \begin{aligned}
    h_{i,j} &= 1 + a(c_{i,j}) + l(c_{i,j}) \\
    &= 1 + \#R_i - j + \#C_j - i \\
    &= 1 + \eta_{l-(i-2)} + i - j + \eta_{l-(j-2)} + (j-1) - i \\
    &= 1 + \#C_i - j + \#R_j - i \\
    &= 1 + l(c_{j,i}) + a(c_{j,i}) \\
    &= h_{j,i}.
    \end{aligned}
    \]
\end{enumerate}
\end{proof}

Let ${\lambda} = S_{2\eta^*}^c$, where $S_{2\eta^*} = \KN^{-1}(Y_{2\eta^*})$. By Proposition~\ref{riconoscimentoNS}, we have for all $i$:
	\begin{itemize}
		\item $h_{i,1} \in \l$;
		\item $h_{1,i} = h_{1,1} - s_{i-1}$, where $s_{i-1} \in S_{2\eta^*}$.
	\end{itemize}
	As a direct consequence of the previous structural lemmas on $Y_{2\eta^*}$, we can now deduce basic properties of the associated partition $\l$, summarized in the following corollary.

\begin{corollary}\label{cor:struct_lambda}
Let $\eta \in \mathbb{D}_k \setminus \{(3,k-3)\}$, and let $\l$ denote the partition obtained from $\eta$ as described above. Then the following hold:
\begin{enumerate}
    \item\label{lunghezzalambda} $| \l | = n-2$, where $n=2k-1$;
    \item\label{massimapartelambda} the largest part of $\l$ satisfies $\l_{n-2} = 2n-4$;
    \item\label{partimancantilambda} the number of missing parts of $\l$ is $\#\mathcal{M}_{\l} = n-2$.
\end{enumerate}
\end{corollary}

\begin{proof}
We address each claim separately.
\begin{enumerate}
    \item From Lemma~\ref{rowsColsY}(\ref{C_i'}) we know that $\#C_1=n-2$, and by Proposition~\ref{riconoscimentoNS} we deduce that $\mid\l\mid=n-2$.
    \item From the previous point we have
    $
    \l_{n-2} = h_{1,1} = 2n-4.
    $
    \item From Lemma~\ref{rowsColsY}(\ref{R_i'}) we know that
    \[
    \#\{s\in S_{2\eta^*} \mid 0\le s \le 2n-4\} = \#R_1 = n-1.
    \]
    Clearly, $\#\mathcal{M}_{\l} = \#\{s\in S_{2\eta^*} \mid 0 < s \le 2n-4\} = n-2$.\qedhere
\end{enumerate}
\end{proof}

We now investigate structural properties of the Young diagram $Y_{2\eta^*}$ that will lead us to deduce the unrefinability of the associated partition $\l = S_{2\eta^*}^c$. In particular, we focus first on the configuration of hooks along the main diagonal and on the interplay between rows and columns of $Y_{2\eta^*}$.

\begin{lemma}\label{hooks_diagonal_and_sums}
Let $\eta \in \mathbb{D}_k \setminus \{(3,k-3)\}$ and denote by $\l$ the partition associated with $Y_{2\eta^*}$ as above. Then the following properties hold:

\begin{enumerate}
    \item \label{hookdiagonale'} for $2 \le i \le l+1$, the hook lengths along the main diagonal satisfy
    \begin{equation}
        h_{i,i} = \l_{n-2} - 2 s_{i-1}; 
    \end{equation}
 \item \label{hooks_diagonal_and_sums2} if $s_i, s_j \neq n-2$, then $s_i + s_j \neq \l_{n-2}$. Consequently, for each $\l_i \in \l$, there exists a unique $s_j \in S_{2\eta^*}$ such that
    \begin{equation}
        \l_i + s_j = \l_{n-2}.
    \end{equation}
    In other words, every part of $\l$ pairs with an element of $S_{2\eta^*}$ to sum to the maximal part.
\end{enumerate}
\end{lemma}

\begin{proof}
We  prove each claim separately.
\begin{enumerate}
    \item By Eq.~\eqref{equazionehookinterno} and Lemma~\ref{hookSymmetrySquare}(\ref{C_1R_1q}), we have
    \begin{align*}
        h_{i,i} &= 2 h_{1,i} - h_{1,1} = \l_{n-2} - 2 s_{i-1}.
    \end{align*}

    \item Suppose, by contradiction, that there exist $s_i, s_j \neq n-2$ such that $\l_{n-2} = s_i + s_j$. Then
    \begin{align*}
        h_{1,i+1} &= \l_{n-2} - s_i = s_j, \\
        h_{1,j+1} &= \l_{n-2} - s_j = s_i.
    \end{align*}
    Without loss of generality, assume $s_i > n-2$ so $s_j < n-2$. Then 
    \[
        h_{j+1,j+1} = \l_{n-2} - 2 s_j > 0,
    \]
    implying $1 \le j \le l$. This would force $h_{j,1} = h_{1,j} \in S_{2\eta^*}$, a contradiction. Hence $s_i + s_j \neq \l_{n-2}$.
    
    Now, let $\l_i \in \l$ be given. From what previously proved, for each $\l_i$ there exists exactly one $s_j \in S_{2\eta^*}$ satisfying $\l_i + s_j = \l_{n-2}$. Existence follows because each $\l_i$ corresponds to a hook in $Y_{2\eta^*}$ that can be ``completed'' to $\l_{n-2}$ by a cell in the first row or first column, which is identified with some $s_j \in S_{2\eta^*}$.  
Uniqueness follows from the fact that no two distinct elements of $S_{2\eta^*}$ (except possibly $n-2$, which is treated separately) can sum to $\l_{n-2}$. Therefore, for each $\l_i \in \l$, there exists a unique $s_j \in S_{2\eta^*}$ such that $\l_i + s_j = \l_{n-2}$.\qedhere
\end{enumerate}
\end{proof}

\begin{corollary}\label{n_minus_2_in_set}
Under the same assumptions of the previous lemma, we have $n-2 \in S_{2\eta^*}$.
\end{corollary}

\begin{proof}
Suppose, by contradiction, that $n-2 \in \l = S_{2\eta^*}^c$. By Lemma~\ref{hooks_diagonal_and_sums}(\ref{hooks_diagonal_and_sums2}), there exists some $s_i \in S_{2\eta^*}$ such that $n-2 + s_i = \l_{n-2}$, forcing $s_i = n-2 \in S_{2\eta^*}$, which is a contradiction.
\end{proof}

Before stating the next lemma, we observe that its the first and third items establish key properties of the column with index $l+2$, while the second reveals a symmetry between the blocks of hooks lying outside the main square and the $l+2$-th column.

\begin{lemma}\label{block2symmetry}
Let $\eta \in \mathbb{D}_k \setminus \{(3,k-3)\}$, and let $Y_{2\eta^*}$ be the associated Young diagram. Then the following hold:
\begin{enumerate}
    \item \label{colonnan-2l-2} 
   $h_{1,l+2} = n-2$;
    \item \label{uguaglianzafq}
    for all $l+2 \le i \le n-2$ and $2 \le j \le l+1$, we have
    \begin{equation*}
        h_{i,1} = h_{1,i+1} \quad \text{and} \quad h_{i,j} = h_{j,i+1};
    \end{equation*}
    \item \label{diag2l+1'}
    for $1 \le i \le l+1$, we have
    \begin{equation*}
        h_{i,i} = 2 h_{i,l+2}.
    \end{equation*}
    In other words,
    \begin{equation*}
        h_{i,l+2} = \eta_{\,l-(i-2)}.
    \end{equation*}
\end{enumerate}
\end{lemma}

\begin{proof}
We prove the claims by considering each item separately.
\begin{enumerate}
    \item From the construction of $Y_{2\eta^*}$, we know $h_{l+1,l+1} = 2 \eta_1 > 0$. Suppose $h_{1,l+2} = 2n-4 - s_{l+1}\neq n-2$, so that $s_{l+1} < n-2$. Then $h_{l+2,1} = 2n-4 - s_{l+1}$ as well. By Eq.~\eqref{equazionehookinterno}, we get
    \begin{equation*}
        h_{l+2,l+2} = h_{1,l+2} + h_{l+2,1} - h_{1,1} = 2(n-2-s_{l+1}) > 0,
    \end{equation*}
    which is impossible by construction.

    \item Since 
		\[
		\#\{h_{i,1} \mid l+2 \leq i \leq n-2\} = \#\{h_{1,j} \mid l+3 \leq j \leq n-1\},
		\]
		and given that $h_{i,1} > h_{i+1,1}$ and $h_{1,i} > h_{1,i+1}$, we can conclude that
		\[
		h_{i,1} = h_{1,i+1}.
		\]

Then, for $2 \le j \le l+1$, we obtain
\[
    \begin{aligned}
			h_{i,j}&=\#R_i-j+\#C_i-j+1\\
			&=h_{i,1}-(n-2-i)-j+h_{1,j}-(n-1-j)-i+1\\
			&=h_{1,i+1}-(n-1-(i+1))-j+h_{j,1}-(n-2-j)-(i+1)+1\\
			&=\#C_{i+1}-j+\#R_j-i+1\\
			&=h_{j,i+1}.
		\end{aligned}
		\]
    \item For $i=1$, $h_{1,1} = 2n-4 = 2(n-2) = 2 h_{1,l+2}$ by item~(\ref{colonnan-2l-2}). For $2 \le i \le l+1$, by Lemma~\ref{hooks_diagonal_and_sums}(\ref{hookdiagonale'}) we have
    \begin{equation*}
        h_{i,i} = h_{1,1} - 2 s_{i-1} = 2(n-2 - s_{i-1}),
    \end{equation*}
    and by Eq.~\eqref{equazionehookinterno} we have
    \begin{equation*}
  h_{i,l+2}=h_{i,1}+h_{1,l+2}-h_{1,1}=h_{1,i}-(n-2)=n-2-s_{i-1}.\qedhere
    \end{equation*}
\end{enumerate}
\end{proof}

\begin{example}
	Let us consider the partition $\eta=(1,3,4)\in\mathbb{D}_8$ and let us apply the Young diagram construction introduced in this section to this partition. 
	Doing so, we recover exactly the diagram shown in Fig.~\ref{fig:ex-final} of Example~\ref{exa:main}. 
	In particular, one can observe the hooks corresponding to the column with index $l+2=5$ (which contains $\eta$), as well as the symmetries among the hooks colored in the same way.
\end{example}

Before concluding that the partition $\l$ obtained from $\eta$ is unrefinable, we first verify that its weight coincides with $T_n$.

\begin{proposition}\label{pesolambda}
	Let $\eta\in\mathbb{D}_k\setminus \{(3,k-3)\}$. Then
	\[
	\l \vdash T_n.
	\]
\end{proposition}

\begin{proof}
	The weight of $\l$ is given by the sum of the hooks in the first column. Hence,
	\[
	\begin{aligned}
		\begin{aligned}
			\sum_{i=1}^{n-2}h_{i,1}&=2n-4+\sum_{i=2}^{n-2}h_{i,1}\\
			&=2n-4+\sum_{i=2}^{n-2}\left(l(c_{i,1})+1\right)+a(c_{i,1})\\
				&=2n-4+\sum_{i=1}^{n-3}i+\sum_{i=2}^{n-2}a(c_{i,1})\\
			&=2n-4+\sum_{i=1}^{n-3}i+\sum_{i=1}^{l}2\eta_i\\
			&=2n-4+\frac{(n-3)(n-2)}{2}+2k\\
			&=\frac{1}{2}(4n-8+n^2-5n+6+2n+2)\\
			&=\frac{1}{2}(n^2+n)\\
			&=T_n.
		\end{aligned}
	\end{aligned}
	\]
\end{proof}
\subsubsection{Proof of unrefinability}\label{subsec:unref}

We now reach the last stage of the argument: proving that the partition $\lambda$
constructed from $\eta$ as in Sec.~\ref{sec:ritornotriang} is unrefinable.  

According to Proposition~\ref{riconoscimentounref}, a partition is unrefinable precisely when
every hook outside the first column satisfies one of the following conditions:
\begin{enumerate}
    \item it coincides with one of the hooks of the first column (Proposition~\ref{riconoscimentounref}(\ref{item-ric3a})),
    \item it is exactly one half of the hook of the cell in the same row (Proposition~\ref{riconoscimentounref}(\ref{item-ric3b})).
\end{enumerate}

Therefore, to prove that $\lambda$ is unrefinable, we proceed by contradiction. 
Namely, assuming that a hook $s_\alpha$ outside the first column does {not} satisfy 
the first condition (i.e.\ $s_\alpha \notin \l=S^c_{2\eta^*}$), we will show that such a hook either necessarily falls into 
the second case, or its existence leads to a contradiction.
To do this, we split the diagram into three natural regions:
\begin{enumerate}
    \item hooks lying inside the main square ($2 \le i,j \le l+1$), i.e.\ the green-colored area in Fig.~\ref{fig:ex-final};
    \item hooks in the shifted column with index $l+2$, i.e.\ the pink-colored area in Fig.~\ref{fig:ex-final};
    \item hooks below the square (and, by symmetry, also those to the right of the main square and of the extra column $l+2$), 
i.e.\ the yellow-colored area in Fig.~\ref{fig:ex-final}.
\end{enumerate}

In the next lemmas we analyze each region separately, denoting by $s_\alpha$ one of the hooks $h_{i,j}$ of $Y_{2\eta^*}$. 

\begin{lemma}\label{salphaidiversoj}
Let $2 \le i,j \le l+1$ with $i \neq j$, and suppose that $s_\alpha$ is a hook length outside the first column, as defined above. Then the existence of such an $s_\alpha$ leads to a contradiction.
\end{lemma}

\begin{proof}
First, by Eq.~\eqref{equazionehookinterno} and Lemma~\ref{hookSymmetrySquare}, the hook length $s_{\alpha}$ can be expressed as
\[
\begin{aligned}
s_\alpha=h_{i,j}&=h_{1,j}+h_{i,1}-h_{1,1}\\
&= h_{1,j}+h_{1,i}-h_{1,1}\\
&= \l_{n-2}-s_{j-1}+\l_{n-2}-s_{i-1}-\l_{n-2}\\
&= \l_{n-2}-s_{j-1}-s_{i-1}.\\
\end{aligned}
\]
For convenience we rewrite this equivalently in terms of the parts of $\eta$:
\begin{equation}\label{eq:boh}
s_{\alpha} = \eta_{l-(i-2)} + \eta_{l-(j-2)}.
\end{equation}
This equivalence follows from the relation 
$
2\eta_{l-(x-2)} = h_{x,x} = \l_{n-2} - 2 s_{x-1}$, so that $s_{x-1} = (n-2) - \eta_{l-(x-2)}$, for $2 \le x \le l+1$.

Now, let us assume $\alpha+1\neq i,j$. Then, since $\eta\vdash k$, we have
		$
		s_{\alpha}=\eta_{l-(i-2)}+\eta_{l-(j-2)}\leq k
		$.
		 For $n>5$, we have $n-2>\frac{n+1}{2}=k$, and then $s_{\alpha}<n-2$. This implies  
		\[
		h_{\alpha+1,\alpha+1}=2\eta_{l-(\alpha-1)}=2n-4-2s_{\alpha}>0
		\]
		and 
		\[
		n-2=\eta_{l-(i-2)}+\eta_{l-(j-2)}+\eta_{l-(\alpha-1)}\leq k,
		\]
		a contradiction.\\
		Without loss of generality, let us now assume $\alpha=i-1$. Then we have 
		\begin{equation}\label{equazioneidiversoj}
			n-2=\eta_{l-(j-2)}+2\eta_{l-(i-2)}.
		\end{equation}
		Computing
		\[
		\eta_{l-(j-2)}+\eta_{l-(i-2)}=\eta_{l-(j-2)}+\frac{n-2-\eta_{l-(j-2)}}{2}=\frac{n-2+\eta_{l-(j-2)}}{2},
		\]
		since $\eta\vdash k$, we obtain
		\[
		n-2+\eta_{l-(j-2)}\leq 2k=n+1.
		\]
		Symplifying,
		\[
		\eta_{l-(j-2)}\leq3.
		\]
		We distinguish now the three possible cases:
		\begin{itemize}
			\item if $\eta_{l-(j-2)} = 1$, then 
			\[
			\eta_{l-(i-2)} = \frac{n-3}{2} = \frac{2k-4}{2} = k-2,
			\]
			but $\eta \vdash k$, hence $\eta = (1,1,k-2)$, a contradiction;  
			\item if $\eta_{l-(j-2)} = 2$, then $\eta_{l-(i-2)} = \frac{2k-5}{2} \notin \mathbb{Z}$, a contradiction; 
			\item if $\eta_{l-(j-2)} = 3$, then $\eta_{l-(i-2)} = k-3$ and $\eta = (3,k-3)$, again a contradiction.\qedhere
		\end{itemize}
	\end{proof}

\begin{lemma}\label{salphai=j}
Let $2 \le i = j \le l+1$, and assume that a hook outside the first column has length $s_\alpha$, as defined above. Then the existence of such an $s_\alpha$ forces the relation
\[
h_{\alpha+1,1} = 2 h_{\alpha+1,\alpha+1}.
\]
\end{lemma}

\begin{proof}
From the hypotheses we have
		\begin{equation}\label{equazionei=jsalpha}
			s_{\alpha} = h_{i,i} = 2\eta_{l-(i-2)}.
		\end{equation}
We proceed by distinguishing three possibile cases,  starting by assuming  $s_{\alpha} < n-2$. In this case
			\[
			\l_{n-2} - 2s_{\alpha}=h_{\alpha+1,\alpha+1} = 2\eta_{l-(\alpha-1)} > 0.
			\]
			Rewriting, we obtain
			\[
			n-2 = s_{\alpha} + \eta_{l-(\alpha-1)}.
			\]
			Using Eq.~\eqref{equazionei=jsalpha}, this gives
			\begin{equation}\label{equazionei=jn-2}
				n-2 = \eta_{l-(\alpha-1)} + 2\eta_{l-(i-2)}.
			\end{equation}
			We distinguish now two sub-cases: 
			\begin{itemize}
				\item if $\alpha \neq i-1$, we obtain the same contradiction given from Eq.~\eqref{equazioneidiversoj}, swapping the role of \(\alpha\) and \(j\);
				\item if $\alpha = i-1$, then Eq.~\eqref{equazionei=jn-2} gives
				$
				n-2 = 3\eta_{l-(\alpha-1)}
				$
				or equivalently
				$
				\l_{n-2} = 6\eta_{l-(\alpha-1)} = 3 s_{\alpha}.
				$
				Hence
				\[
				\begin{aligned}
					&h_{\alpha+1,1} = h_{1,\alpha+1} = \l_{n-2} - s_{\alpha} = 2 s_{\alpha},\\
					&h_{\alpha+1,\alpha+1} = \l_{n-2} - 2 s_{\alpha} = s_{\alpha},
				\end{aligned}
				\]
				from which the claimed equality follows.
			\end{itemize}
			
			 Assume now $s_{\alpha} = n-2$. From Eq.~\eqref{equazionei=jsalpha} we have
			\[
			2\eta_{l-(i-2)} = s_{\alpha} = n-2 = 2k-3,
			\]
			which is impossible since $\eta_{l-(i-2)}$ is an integer.
			
			Finally, assume $s_{\alpha} > n-2$, hence $s_{\alpha} \geq n-1 = 2k-2$. Then, from Eq.~\eqref{equazionei=jsalpha}, we get
			$
			\eta_{l-(i-2)} \geq k-1.
			$
		Since $\eta \vdash k$, we must have $\eta = (1,k-1)$.  
To reconstruct the Young diagram in this case, we determine all relevant
quantities; for convenience, we collect them in Tab.~\ref{fig:contradiction-salpha}.

\begin{table}
\begin{tabular}{c|c}
\hline
\textbf{Parameter} & \textbf{Value} \\
\hline\hline
$l$ & $2$ \\
$i$ & $2$ \\
$\lambda_{n-2}$ & $4k-6$ \\
$s_{\alpha}=h_{2,2}$ & $2k-2$ \\
$s_{i-1}$ & $k-2$ \\
$h_{1,\alpha+1}$ & $2k-4$ \\
$h_{i,\alpha+1}$ & $k-2$\\
\hline
\end{tabular}
\medskip

\caption{Parameters for the diagram of the partition $(1,k-1)$ in the proof of Lemma~\ref{salphai=j}.}
\label{fig:contradiction-salpha}
\end{table}

Observing that $\alpha = 4$, we attempt to reconstruct the Young diagram $Y_{{2\eta^*}}$ under the assumption $\eta = (1, k-1)$, as illustrated in Fig.~\ref{fig:attempt-diagram}, with the relevant hook lengths and parameters indicated. Following the construction, we compute
\[
k-1 = h_{2,4} = h_{2,5}+2 = k,
\]
which immediately leads to a contradiction.\qedhere

\begin{figure}
			\[
			\begin{tikzpicture}
				\draw[color=black, fill=red!30] (0,0) rectangle (1.5,-0.5) rectangle (3,-1) rectangle (4.5,-1.5);
				\draw[color=black, fill=pink!30]  (4.5,-1.5) rectangle (6,-1) rectangle (4.5,-0.5) rectangle (6,0);
				\draw[color=black, fill=yellow!30]  (6,0) rectangle (7.5,-0.5) rectangle (6,-1);
			%	\draw (0,0) rectangle (1.5,-0.5) rectangle (3,-1) rectangle (4.5,-1.5) rectangle (6,-1) rectangle (4.5,-0.5) rectangle (6,0) rectangle (7.5,-0.5) rectangle (6,-1);
				\node at (0.75,-0.25) {$4k-6$};
				\node at (2.25,-0.75) {$2k-2$};
				\node at (3.75,-1.25) {$2$};
				\node at (5.25,-1.25) {$1$};
				\node at (5.25,-0.75) {$k-1$};
				\node at (5.25,-0.25) {$2k-3$};
				\node at (0.75,0.25) {\scriptsize{$1$}};
				\node at (-0.25,-0.25) {\scriptsize{$1$}};
				\node at (5.25,0.25) {\scriptsize{$l+2=4$}};
				\node at (6.75,-0.75) {$k-2$};
				\node at (6.75,-0.25) {$2k-4$};
	
				\draw[pattern=north east lines, pattern color=black!55, dashed] 
      			(4.5,-1.5) rectangle (3,-2);
      			\draw[pattern=north east lines, pattern color=black!55, dashed] 
      			(4.5,-1.5) rectangle (6,-2);
      			\draw[pattern=north east lines, pattern color=black!55, dashed] 
      			(6,-1) rectangle (7.5,-1.5);
				%\node at (6.75,-1.25) {X};
				%\node at (5.25,-1.75) {X};
				%\node at (3.75,-1.75) {X};
				\node at (9,-0.25) {$\cdots$};
				\node at (0.75,-2) {$\vdots$};
			\end{tikzpicture}
			\]
			  \caption{Attempted reconstruction of the Young diagram $Y_{{2\eta^*}}$ for $\eta = (1,k-1)$ showing the relevant hook lengths.}
    \label{fig:attempt-diagram}
			\end{figure}
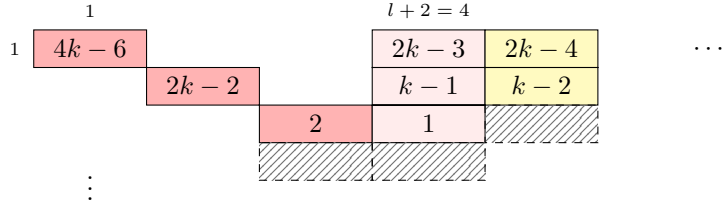
\end{proof}

We now focus on the elements located outside the main square.

\begin{lemma}\label{lemmaboh}
    Let $1 \le j \le l+1$ and $l+2 \le i \le n-2$. Then the existence of such an $s_{\alpha}$ leads to a contradiction.
\end{lemma}
\begin{proof}
By Lemma~\ref{block2symmetry}(\ref{uguaglianzafq}) we have
    \[
        h_{i,j} = h_{j,i+1} = \l_{n-2} - s_{j-1} - s_i.
    \]
    Since $s_i > n-2$, it follows that
    \[
        s_{\alpha} = h_{i,j} < n-2.
    \]
    Hence $1 \le \alpha \le l+1$, and $h_{\alpha+1,j} = s_i$, which reduces to a situation already treated in Lemma~\ref{salphaidiversoj} and the third case of Lemma~\ref{salphai=j}, leading to a contradiction.
\end{proof}

\begin{lemma}\label{lemmaboh2}
    Let $2 \le i \le l+1$ and $j = l+2$. Then the existence of such an $s_{\alpha}$ leads to a contradiction.
\end{lemma}
\begin{proof}
    By Lemma~\ref{block2symmetry}(\ref{diag2l+1'}), we obtain
    \[
        s_{\alpha} = h_{i,l+2} = \eta_{l-(i-2)} < k < n-2.
    \]
    Therefore,
    \[
        h_{\alpha+1,\alpha+1} = \l_{n-2} - 2 s_{\alpha} > 0,
    \]
which becomes    \[
        s_{\alpha} + \eta_{l-(\alpha-1)} = n-2.
    \]
    Simplifying, we get
    \[
        n-2 = \eta_{l-(\alpha-1)} + \eta_{l-(i-2)} \le k < n-2,
    \]
    which is a contradiction.
\end{proof}
	
	All the preceding lemmas collectively establish that any hook length $s_\alpha$ outside the first column of $Y_{{2\eta^*}}$ either coincides with the length of a hook in the first column or is exactly half of the hook length of the corresponding cell in the first column. Consequently, all possible positions of a hook outside the first column have been considered, and any assumption violating the criterion of Proposition~\ref{riconoscimentounref} leads to a contradiction. This allows us to summarize our analysis in the following proposition.

\begin{theorem}\label{lambdanraffinabile}
    Let $\eta\in\mathbb{D}_k\setminus\{(3,k-3)\}$, and let $\lambda = S^c_{2\eta^*}$ be the partition associated to the Young diagram $Y_{{2\eta^*}}$, where $S_{2\eta^*} = \KN^{-1}(Y_{2\eta^*})$. Then $\lambda$ is unrefinable. In particular $\l\in\overline{\mathcal{U}}_{T_n}$.
\end{theorem}
\begin{proof}
The proof  follows directly from the previous results. In particular: 
$\lambda$ is unrefinable by the above lemmas, 
satisfies the weight condition for $T_n$ by Proposition~\ref{pesolambda}, 
is maximal among the partitions of $T_n$ by Corollary~\ref{cor:struct_lambda}(\ref{massimapartelambda}), 
and finally satisfies the necessary conditions for belonging to 
$\overline{\mathcal{U}}_{T_n}$ according to Corollary~\ref{cor:struct_lambda}(\ref{partimancantilambda}).
\end{proof}

To conclude, observe that, for $k\ge 4$, the previous results immediately give that the number of maximal unrefinable partitions of $T_{2k-1}$ coincides with $\#\mathbb{D}_k$, which in turn equals $\#\mathbb{D}_{2k}^e$. 
This count is obtained by associating the partition $(3,k-3)\in\mathbb{D}_k$, which we excluded, with the partition $\wtp$, which was also excluded. 
The following example illustrates why the correspondence $\KN^{-1}$ behaves in a nonconforming way on the partition $(3,k-3)$.

\begin{example}\label{controesempio}
Let $k=8$, $n=2k-1= 15$. Consider $\eta = (3,k-3) = (3,5)$. Notice that the Young diagram $Y_{2\eta^*}$ does not respect the requirements of unrefinability of Proposition~\ref{riconoscimentounref}, how it can be deduced by the highlighted boxes in Fig.~\ref{fig:young-3k-3}. 
\end{example}
\begin{figure}
\begin{tikzpicture}
			\draw (0,0) rectangle (0.5, -0.5) rectangle (0,-1) rectangle (0.5, -1.5) rectangle (0,-2) rectangle (0.5,-2.5) rectangle (0,-3)
			rectangle (0.5, -3.5) rectangle (0,-4) rectangle (0.5, -4.5) rectangle (0,-5) rectangle (0.5, -5.5) rectangle (0,-6) rectangle 				(0.5, -6.5);
			\draw (1,0) rectangle (0.5, -0.5) rectangle (1,-1) rectangle (0.5, -1.5) rectangle (1,-2) rectangle (0.5,-2.5) rectangle (1,-3);
			\draw (1,0) rectangle (1.5, -0.5) rectangle (1,-1) rectangle (1.5, -1.5) rectangle (1,-2) rectangle (1.5,-2.5);
			\draw (2,0) rectangle (1.5, -0.5) rectangle (2,-1) rectangle (1.5,-1.5);
			\draw (2,0) rectangle (2.5, -0.5) rectangle (2,-1) rectangle (2.5,-1.5);
			\draw (3,0) rectangle (2.5, -0.5) rectangle (3,-1) rectangle (2.5,-1.5);
			\draw (3,0) rectangle (3.5, -0.5) rectangle (3,-1);
			\draw (3.5,-0.5) rectangle (4,0) rectangle (4.5,-0.5) rectangle (5,0) rectangle (5.5,-0.5) rectangle (6,0) rectangle (6.5,-0.5) rectangle (7,0) ;
			\node at (0.25,-0.25) {$26$};
			\node at (0.25,-0.75) {$18$};
			\node at (0.25,-1.25) {$16$};
			\node at (0.25,-1.75) {$12$};
			\node at (0.25,-2.25) {$11$};
			\node at (0.25,-2.75) {$9$};
			\node at (0.25,-3.25) {$7$};
			\node at (0.25,-3.75) {$6$};
			\node at (0.25,-4.25) {$5$};
			\node at (0.25,-4.75) {$4$};
			\node at (0.25,-5.25) {$3$};
			\node at (0.25,-5.75) {$2$};
			\node at (0.25,-6.25) {$1$};
			
			\node at (0.75,-0.25) {$18$};
			\node [color=purple]  at (0.75,-0.75) {$10$};
			\node at (0.75,-1.25) {$8$};
			\node at (0.75,-1.75) {$4$};
			\node at (0.75,-2.25) {$3$};
			\node at (0.75,-2.75) {$1$};
			
			\node at (1.25,-0.25) {$16$};
			\node [color=purple]  at (1.25,-0.75) {$8$};
			\node at (1.25,-1.25) {$6$};
			\node at (1.25,-1.75) {$2$};
			\node at (1.25,-2.25) {$1$};
			
			\node at (1.75,-0.25) {$13$};
			\node at (1.75,-0.75) {$5$};
			\node at (1.75,-1.25) {$3$};
			
			\node at (2.25,-0.25) {$12$};
			\node at (2.25,-0.75) {$4$};
			\node at (2.25,-1.25) {$2$};
			
			\node at (2.75,-0.25) {$11$};
			\node at (2.75,-0.75) {$3$};
			\node at (2.75,-1.25) {$1$};
			
			\node at (3.25,-0.25) {$9$};
			\node at (3.25,-0.75) {$1$};
			
			\node at (3.75,-0.25) {$7$};
			\node at (4.25,-0.25) {$6$};
			\node at (4.75,-0.25) {$5$};
			\node at (5.25,-0.25) {$4$};
			\node at (5.75,-0.25) {$3$};
			\node at (6.25,-0.25) {$2$};
			\node at (6.75,-0.25) {$1$};
			
		\end{tikzpicture}
	\caption{Young diagram for the partition $(3,k-3)$ illustrating the nonconforming behavior of $\KN^{-1}$.}
    \label{fig:young-3k-3}
\end{figure}

We are now ready to turn to the nontriangular case. In this setting, addressed in the remaining sections of this work, the structure of the unrefinable partitions of $T_{n,d}$ naturally splits into two distinct families, according to the value of the maximal part of the partition, namely $2n-5$, when $n-d$ is even, or $2n-4$, when $n-d$ is odd. In Sec.~\ref{sec:nontri5} we deal with the case of $T_{n,n-2k}$ and in Sec.~\ref{sec:nontri4} we address the case of $T_{n,n-2k+1}$.

\section{Nontriangular weight with maximal part $2n-5$}\label{sec:nontri5}

In this section, we aim to prove the bijection 
\(\mup{T_{n,n-2k}} \leftrightarrow \mathbb{D}^{\,\text{o}}_{2k+2}\), where $2\leq 2k\leq n-4$.  
We establish this correspondence by constructing again a double embedding, presented in the two subsections that follow.
As in the triangular case, we must first remove from $\widetilde{\mathcal{U}}_{T_{n,n-2k}}$ the unique partition that does not attain the maximal number of missing parts, namely the exceptional partition $\widetilde{\tau}_n$ listed in Tab.~\ref{tab:exceptions} (existing only when $2k= n-4$, or $d=4$).  

\subsection{Proof of \(\mup{T_{n,n-2k}} \hookrightarrow  \mathbb{D}^{\,\text{o}}_{2k+2}\)}
Let $\lambda = (\lambda_1, \dots, \lambda_t) \in \overline{\mathcal{U}}_{T_{n,n-2k}} = \mup{T_{n,n-2k}} \setminus \{\widetilde{\tau}_n\}$  and $2\leq 2k\leq n-4$, and let $S_\lambda \in \NS$ be the numerical set such that $S_\lambda^c = \lambda$. We consider the Young diagram
\[
Y_{S_\lambda} = \KN(S_\lambda),
\]
constructed from $S_\lambda$ via the Keith-Nath procedure, as before.

We first establish a general structural property: when the maximal part is odd, the set of parts of an unrefinable partition is completely determined by its missing parts, via a perfect complementarity relation. 
\begin{lemma}\label{lem:sym}
Let $\lambda=(\lambda_1,\dots,\lambda_t)\in\overline{\mathcal{U}}$ and suppose that $\lambda_t$ is odd. Then
\[
\lambda_i\in\lambda \quad\Longleftrightarrow\quad \lambda_t-\lambda_i\notin\lambda.
\]
\end{lemma}

\begin{proof}
Write $\lambda_t=2q-1$. Since $\lambda\in\overline{\mathcal{U}}$, we have
\[
\#\mathcal{M}_\lambda=\left\lfloor\frac{\lambda_t}{2}\right\rfloor=q-1,
\qquad
|\lambda|=q.
\]

Recall that for every missing part $\mu\in\mathcal{M}_\lambda$ there must exist a part $\lambda_i\in\lambda$ such that
$
\lambda_i+\mu=\lambda_t,
$
otherwise $\lambda$ would be refinable.
Moreover,
$
\#\{\lambda_i\in\lambda \mid \lambda_i<\lambda_t\}
=\#\mathcal{M}_\lambda,
$
so the pairing between present parts and missing parts is bijective.  
This yields the desired equivalence.
\end{proof}

Before stating the main structural facts, let us comment on their geometric interpretation.
Just as in the triangular case, the behaviour of the entries in the first row and the first column
governs the entire Young diagram.  
Here the key phenomenon is that, when the maximal part is $2n-5$, the diagram becomes
\emph{symmetric}: the first row and the first column contain the same number of cells and,
moreover, the hook-lengths computed along them match perfectly.
This symmetry propagates to the whole diagram, forcing the identity
\[
h_{i,j}=h_{j,i},
\]
which is the geometric counterpart of the complementarity relations satisfied by the parts of~$\lambda$.

\begin{lemma}\label{lem:nontriangLR}
Using the above notation, the following claims hold:
\begin{enumerate}
    \item\label{C1:nontriang}
    the first column of $Y_{S_\lambda}$ contains $n-2$ cells, and the hook length of the cell in the $i$-th row of this column is
    \[
    h_{i,1} = \lambda_{\,n-2-(i-1)};
    \]

    \item\label{R1:nontriang}
    the first row of $Y_{S_\lambda}$ contains $n-2$ cells, and the hook length of the cell in the $i$-th column of this row is
    \[
    h_{1,i} = \lambda_{\,n-2}-s_{i-1}.
    \]
\end{enumerate}
\end{lemma}

\begin{proof}
The statements follow directly from the proofs of Lemma~\ref{lemLR}.
\end{proof}

\begin{lemma}\label{lem:simmetriaR1C1-nontriang}
For every $1 \le i \le n-2$, we have
\[
h_{i,1}=h_{1,i}.
\]
\end{lemma}

\begin{proof}
For $i=1$ the statement is trivial.  
For $i>2$, since
\[
h_{1,i} = \lambda_{n-2} - s_{i-1},
\]
this value equals a part of $\lambda$ (see Lemma~\ref{lem:sym}), hence appears among the hook lengths in the first column.  
Since by Lemma~\ref{lem:nontriangLR} we have $\#C_1=\#R_1=n-2$, and both sequences
$(h_{i,1})_i$ and $(h_{1,i})_i$ are strictly decreasing, they must coincide termwise.  
Thus,
\[
h_{i,1}=h_{1,i}.\qedhere
\]
\end{proof}

\begin{corollary}\label{cor:simmetriaToda}
For all $1\le i,j\le n-2$, we have
\[
h_{i,j} = h_{j,i}.
\]
\end{corollary}

\begin{proof}
Using the standard hook-length identity (Eq.~\eqref{equazionehookinterno}),
\[
h_{i,j}=h_{i,1}+h_{1,j}-h_{1,1},
\]
and applying Lemma~\ref{lem:simmetriaR1C1-nontriang}, we obtain
\[
h_{i,j}
= h_{1,i}+h_{j,1}-h_{1,1}
= h_{j,i}.\qedhere
\]
\end{proof}

The symmetry properties established in the previous lemmas imply that the Young
diagram associated with any $\lambda\in\overline{\mathcal{U}}_{T_{n,n-2k}}$ is perfectly
reflected across its main diagonal. In other words, $Y_{S_\lambda}$ is
\emph{self-conjugate}: the number of cells in the $i$-th row equals the number
of cells in the $i$-th column, and every hook length satisfies
$h_{i,j}=h_{j,i}$. Equivalently, the diagram coincides with its transpose, so
that its entire shape is determined by its first row (or, dually, its first
column). This is particularly relevant because self-conjugate Young diagrams are in
one-to-one correspondence with partitions into distinct odd parts~\cite{andrews1998theory}. Therefore,
the self-conjugacy of $Y_{S_\lambda}$ establishes a direct link between maximal
unrefinable partitions in $\overline{\mathcal{U}}_{T_{n,n-2k}}$ and this classical
family of integer partitions.\\

Now we can count the number of cells of $Y_{S_\lambda}$.

\begin{proposition}
    Let $\lambda \in \overline{\mathcal{U}}_{T_{n,n-2k}}$. Then the total number of cells in $Y_{S_\lambda}$ is 
    \[
    \#Y_{S_\lambda} = 2n - 3 + 2k.
    \]
    Consequently, the sum of the hook lengths along the main diagonal is $2n - 3 + 2k$.
\end{proposition}

\begin{proof}
    The total number of cells in the diagram can be computed as the sum of the number of cells in each row:
    \[
    \begin{aligned}
        \#Y_{S_\lambda} &= \sum_{i=1}^{n-2} \#R_i \\
        &= \sum_{i=1}^{n-2} h_{i,1} - (\#C_1 - i) \\
        &= \sum_{i=1}^{n-2} h_{i,1} - \sum_{i=1}^{n-2} \#C_1 + \sum_{i=1}^{n-2} i \\
        &= T_{n,n-2k} - (n-2)(n-2) + T_{n-2} \\
        &= \frac{n(n+1)}{2} - (n-2k) - (n-2)^2 + \frac{(n-2)(n-1)}{2} \\
        &= \frac{1}{2}(4n - 6 + 4k) \\
        &= 2n - 3 + 2k.
    \end{aligned}
    \]
\end{proof}

\begin{remark}
The results above can be interpreted as providing an embedding of the set $\overline{\mathcal{U}}_{T_{n,n-2k}}$ into the set of self-conjugate Young diagrams, i.e. into $\mathbb{D}^{\, \text{o}}_{2k+2}$. This is because every partition $\l$ corresponds via the $\KN$-transformation  to self-conjugate partitions of weight $2n-3+2k$. Such partitions share the same maximal part, namely $2n-5$, that we can remove to obtain a bijective correspondence with $\mathbb{D}^{\, \text{o}}_{2k+2}$. 
From this perspective, the contribution of each part can be read directly along the main diagonal of $Y_{S_\lambda}$.
\end{remark}
We conclude with an example where all the symmetry properties of the Young diagram discussed above are clearly visible.

\begin{example}
		We consider the partition $\l=(1,2,3,4,5,6,7,8,9,12,14,15,25)\in \overline{\mathcal U}_{T_{15,9}}$, i.e.\ obtained for $n=15$ and $k=3$. The resulting Young diagram $Y_{\l}$ is displayed in Fig.~\ref{fig:youngsym}. Notice the symmetry highlighted by green boxes, and the corresponding distinct-odd-part partition $(3,5)$ of $2k+2=8$ in the red boxes.

	\begin{figure}
\begin{tikzpicture}
			\draw [color=black, fill=ForestGreen!30](0,0) rectangle (0.5, -0.5) rectangle (0,-1) rectangle (0.5, -1.5) rectangle (0,-2) rectangle (0.5,-2.5) rectangle (0,-3)
			rectangle (0.5, -3.5) rectangle (0,-4) rectangle (0.5, -4.5) rectangle (0,-5) rectangle (0.5, -5.5) rectangle (0,-6) rectangle 				(0.5, -6.5);
			\draw [color=black, fill=ForestGreen!30](1,0) rectangle (0.5, -0.5) rectangle (1,-1) rectangle (0.5, -1.5) rectangle (1,-2);
			\draw [color=black, fill=ForestGreen!30](1,0) rectangle (1.5, -0.5) rectangle (1,-1) rectangle (1.5, -1.5) rectangle (1,-2);
			\draw [color=black, fill=ForestGreen!30](2,0) rectangle (1.5, -0.5) rectangle (2,-1) rectangle (1.5,-1.5);
			\draw [color=black, fill=ForestGreen!30](2,0) rectangle (2.5,-0.5) rectangle (3,0) rectangle (3.5,-0.5) rectangle (4,0) rectangle (					4.5,-0.5) rectangle (5,0) rectangle (5.5,-0.5) rectangle (6,0) rectangle (6.5,-0.5);
			\draw [color=black, fill=red!30] (0,0) rectangle (0.5, -0.5);
			\draw [color=purple, fill=red!30, line width=1.1pt] (0.5, -0.5)rectangle (1,-1) rectangle (1.5,-1.5);
			\node at (0.25,-0.25) {$25$};
			\node at (0.25,-0.75) {$15$};
			\node at (0.25,-1.25) {$14$};
			\node at (0.25,-1.75) {$12$};
			\node at (0.25,-2.25) {$9$};
			\node at (0.25,-2.75) {$8$};
			\node at (0.25,-3.25) {$7$};
			\node at (0.25,-3.75) {$6$};
			\node at (0.25,-4.25) {$5$};
			\node at (0.25,-4.75) {$4$};
			\node at (0.25,-5.25) {$3$};
			\node at (0.25,-5.75) {$2$};
			\node at (0.25,-6.25) {$1$};
			
			\node at (0.75,-0.25) {$15$};
			\node at (0.75,-0.75) {$5$};
			\node at (0.75,-1.25) {$4$};
			\node at (0.75,-1.75) {$2$};
			
			\node at (1.25,-0.25) {$14$};
			\node at (1.25,-0.75) {$4$};
			\node at (1.25,-1.25) {$3$};
			\node at (1.25,-1.75) {$1$};
			
			\node at (1.75,-0.25) {$12$};
			\node at (1.75,-0.75) {$2$};
			\node at (1.75,-1.25) {$1$};
			
			\node at (2.25,-0.25) {$9$};
			\node at (2.75,-0.25) {$8$};
			\node at (3.25,-0.25) {$7$};
			\node at (3.75,-0.25) {$6$};
			\node at (4.25,-0.25) {$5$};
			\node at (4.75,-0.25) {$4$};
			\node at (5.25,-0.25) {$3$};
			\node at (5.75,-0.25) {$2$};
			\node at (6.25,-0.25) {$1$};	
		\end{tikzpicture}
	\caption{Young diagram of a maximal unrefinable partition illustrating
the main diagonal and its symmetries.}
    \label{fig:youngsym}
\end{figure}

	\end{example}
In the following section, we  complete the proof of the bijection, i.e.\ we show how each element of $\overline{\mathcal{U}}_{T_{n,n-2k}}$ can be uniquely associated with a partition in $\mathbb{D}^{\, \text{o}}_{2k+2}$, except for the partition $(1,2k+1)$ (which is ruled out only when the parameters satisfy $2k=n-4$).
Since 

\[
\#\mup{T_{n,n-2k}} =
\begin{cases}
\#\overline{\mathcal{U}}_{T_{n,n-2k}} + 1 & \text{if } 2k = n-4,\\[6pt]
\#\overline{\mathcal{U}}_{T_{n,n-2k}} & \text{otherwise},
\end{cases}
\]
this discrepancy accounts exactly for the exceptional partition $\widetilde{{\tau}}_n$ in Tab.~\ref{tab:exceptions}, yielding the desired correspondence.

In this sense, it will be useful to define

\[
\overline{\mathbb{D}}^{\, \text{o}}_{2k+2} \deq
\begin{cases}
{\mathbb{D}}^{\, \text{o}}_{2k+2}\setminus\{(1,2k-1)\} & \text{if } 2k = n-4,\\[6pt]
{\mathbb{D}}^{\, \text{o}}_{2k+2} & \text{otherwise}.
\end{cases}
\] 

\subsection{Proof of \( \mathbb{D}^{\text{o }}_{2k+2} \hookrightarrow \mup{T_{n,n-2k}} \)}

Let $\eta=(\eta_1,\cdots,\eta_l)\in\dbar$, and let us define $\eta^*=(\eta_1,\cdots,\eta_l,2n-5)$, where $2 \leq 2k \leq n-4$.

We construct the Young diagram $Y_{\eta^*}$ by imposing the following conditions:
\begin{align}
h_{1,1} &= 2n - 5, \nonumber \\
h_{i,i} &= \eta_{\,l-(i-2)}, \qquad 2 \le i \le l+1, \nonumber \\
a(c_{i,i}) &= l(c_{i,i}), \qquad 1 \le i \le l+1. \label{2n-5equazionecostruzione'}
\end{align}

From the Young diagram $Y_{\eta^*}$ constructed above, in the following we will apply the inverse of the $\KN$ transformation, obtaining a numerical set 
$
S_{\eta^*} = \KN^{-1}(Y_{\eta^*})
$.
We will define the associated partition ${\lambda} = S_{\eta^*}^c$, which we will show to be unrefinable.

By construction, the Young diagram $Y_{\eta^*}$ associated to $\eta \in \dbar$ is {self-conjugate}, meaning that the hook lengths are symmetric with respect to the main diagonal. That is, for every pair of indices $i,j$ we have
\[
h_{i,j} = h_{j,i}.
\]

\begin{lemma}\label{2n-5C1R1}
Let $\eta \in \dbar$. Then the first column of $Y_{\eta^*}$ contains 
$
\#C_1 = n-2
$ 
cells, and by self-conjugacy, the first row also contains 
$
\#R_1 = n-2
$ 
cells.
\end{lemma}

\begin{proof}
For the first column, we have
\[
\#C_1 = l(c_{1,1}) + 1 = \frac{h_{1,1}-1}{2} = n-2.
\]
Since $Y_{\eta^*}$ is self-conjugate, the number of cells in the first row coincides with the first column:
\[
\#R_1 = \#C_1 = n-2.\qedhere
\]
\end{proof}

The next corollary provides a geometric interpretation of the partition $\l$ in terms of its associated Young diagram $Y_{\eta^*}$. Specifically, the total number of parts corresponds to the number of cells in the first column, the largest part is given by the top-left hook length, and the number of missing parts reflects the first-row structure, highlighting the diagram's self-conjugate symmetry. We use again that, by Proposition~\ref{riconoscimentoNS}, we have for all $i$:
	\begin{itemize}
		\item $h_{i,1} \in \l$;
		\item $h_{1,i} = h_{1,1} - s_{i-1}$, where $s_{i-1} \in S_{\eta^*}$.
	\end{itemize}

\begin{corollary}\label{cor:struct_lambda_n2k}
Let $\eta \in \dbar$, and let $\lambda$ denote the partition associated to $\eta$ as described above. Then the following hold:
\begin{enumerate}
    \item\label{lunghezzalambda_n2k} the total number of parts of $\lambda$ is $|\lambda| = n-2$;
    \item\label{massimapartelambda_n2k} the largest part of $\lambda$ satisfies $\lambda_{n-2} = 2n-5$;
    \item\label{partimancantilambda_n2k} the number of missing parts of $\lambda$ is $\#\mathcal{M}_{\lambda} = n-3$.
\end{enumerate}
\end{corollary}

\begin{proof}
We prove each claim separately.
\begin{enumerate}
    \item By Lemma~\ref{2n-5C1R1}, the first column of $Y_{\eta^*}$ contains $\#C_1 = n-2$ cells. By the construction of $\lambda$ from $\eta$, this implies that the total number of parts of $\lambda$ is $|\lambda| = n-2$.
    \item From the previous point, we have that the largest part of $\lambda$ corresponds to the hook length of the top-left cell of the diagram:
    $
    \lambda_{n-2} = h_{1,1} = 2n-5.
    $
    \item From Lemma~\ref{2n-5C1R1} we know that
    \[
    \#\{s\in S_{\eta^*} \mid 0\le s \le 2n-5\} = \#R_1 = n-2.
    \]
    Clearly, $\#\mathcal{M}_{\l} = \#\{s\in S_{\eta^*} \mid 0 < s \le 2n-5\} = n-3$. \qedhere
    \end{enumerate}
\end{proof}

\begin{proposition}\label{2n-5peso^*}
Let $\eta \in \dbar$. Then the corresponding partition $\lambda$ satisfies
\[
\lambda \vdash T_{n,n-2k},
\]
where $2 \le 2k \le n-4$.
\end{proposition}

\begin{proof}
The weight of $\lambda$ is equal to the sum of the hook lengths in the first column of $Y_{\eta^*}$. Hence, we have
\[
\begin{aligned}
\sum_{i=1}^{n-2} h_{i,1} 
&= h_{1,1} + \sum_{i=2}^{n-2} h_{i,1} \\
&= (2n-5) + \sum_{i=2}^{n-2} \bigl( l(c_{i,1}) + 1 + a(c_{i,1}) \bigr) \\
&= (2n-5) + \sum_{i=1}^{n-3} i + \sum_{i=1}^{l}  2\eta_i \\
&= 2n-5 + \frac{(n-3)(n-2)}{2} + (2k+2) \\
&= \frac{1}{2}(n^2 - n) + 2k \\
&=\frac{1}{2}(n^2+n)-n+2k\\
&= T_{n,n-2k}. 
\end{aligned}
\]
\end{proof}

\subsubsection{Proof of unrefinability}
To prove that $\lambda$ is unrefinable, we proceed as in the triangular case. 
Assuming by contraposition that a hook $s_\alpha$ outside the first column does not belong to $\l=S^c_{2\eta^*}$, 
we show that it either satisfies the condition of Proposition~\ref{riconoscimentounref}(\ref{item-ric3b}), or its existence leads to a contradiction. 
For this purpose, we distinguish two regions of the diagram: hooks inside the main square and hooks below or to the right of the square.
To proceed, we need the following intermediate result.

\begin{lemma}\label{2n-5hookbound}
Let $c_{i,j}$ be any cell in $Y_{\eta^*}$ with $i,j\ge 2$. Then
\[
h_{i,j} < n-2.
\]
\end{lemma}

\begin{proof}
For $i,j\ge 2$ we have
\begin{equation}\label{eq:boh2}
\eta_{l-(i-2)} = h_{i,i} =  \l_{n-2} - 2 s_{i-1}. 
\end{equation}

In particular, the hook of the top-left interior cell satisfies 
$
h_{2,2} = \eta_l \le 2k+2 < n-2,
$
where we use that $2k < n-2$. Monotonicity of hooks, $h_{i,j} > h_{i+1,j}$ and $h_{i,j} > h_{i,j+1}$, implies that every interior cell $(i,j)$ with $i,j\ge 2$ also satisfies $h_{i,j} < n-2$.
\end{proof}

In the next lemmas we analyze each region separately, denoting by $s_\alpha$ one of the hooks $h_{i,j}$ of $Y_{\eta^*}$. 

\begin{lemma}\label{2n-5idiversoj}
Let $2 \le i,j \le l+1$ with $i \neq j$, and let $s_\alpha$ be a hook length outside the first column. Then the existence of such an $s_\alpha$ leads to a contradiction.
\end{lemma}

\begin{proof}
We proceed by examining  two different situation.
In the first case we assume $\alpha\neq i-1,j-1$. We have
			$
			s_{\alpha}=h_{i,j}=\l_{n-2}-s_{i-1}-s_{j-1},
			$
			or
			\begin{equation}\label{2n-5sasisj=lambda}
				\l_{n-2}=s_{i-1}+s_{j-1}+s_{\alpha}.
			\end{equation}
			From Eq.~\eqref{eq:boh2} in Lemma~\ref{2n-5hookbound}, we obtain
			\[
			s_{\alpha}=\frac{\l_{n-2}-\eta_{l-(\alpha-1)}}{2}.
			\]
			Moreover 
			$s_{i-1}, s_{j-1}<n-2$, since $i,j\leq l+1$. Hence
			\[
			\begin{aligned}
				s_{i-1}&=\frac{\l_{n-2}-\eta_{l-(i-2)}}{2},\\
				s_{i-j}&=\frac{\l_{n-2}-\eta_{l-(j-2)}}{2}.
			\end{aligned}
			\]
			From Eq.~\eqref{2n-5sasisj=lambda}, we obtain
			\[
			\l_{n-2}=\frac{3}{2}\l_{n-2}-\frac{1}{2}(\eta_{l-(i-2)}+\eta_{l-(j-2)}+\eta_{l-(\alpha-1)}). 
			\]
			Simplifying, we have
			\[
			2n-5=\eta_{l-(i-2)}+\eta_{l-(j-2)}+\eta_{l-(\alpha-1)}\leq2k+2\leq n-2,
			\]
			that it is true for $n\leq3$, a contradiction.

Without loss of generality, we now suppose $\alpha=i-1$. Eq.~\eqref{2n-5sasisj=lambda} becomes
			\[
			\l_{n-2}=s_{j-1}+2s_{i-1},
			\]
			from which we obtain
			\[
			s_{j-1}=\l_{n-2}-2s_{i-1}=h_{i,i}=\eta_{l-(i-2)}.
			\]
			Since $s_{j-1}<n-2$, we have 
			\[
			\l_{n-2}=\eta_{l-(j-2)}+2s_{j-1}=\eta_{l-(j-2)}+2\eta_{l-(i-2)}.
			\]
			We can finally compute 
			\begin{equation}\label{2n-5equazioneperquestolemma}
				\eta_{l-(i-2)}+\eta_{l-(j-2)}=\frac{\l_{n-2}-\eta_{l-(j-2)}}{2}+\eta_{l-(j-2)}=n-2+\frac{\eta_{l-(j-2)}-1}{2}.
			\end{equation}
			We distinguish now two sub-cases:
			\begin{itemize}
				\item if $2k+2<n-2$, Eq.~\eqref{2n-5equazioneperquestolemma} becomes
				\[
				n-2+\frac{\eta_{l-(j-2)}-1}{2}<n-2,
				\]
				clearly a contradiction;
				\item if $2k+2=n-2$, from Eq.~\eqref{2n-5equazioneperquestolemma}, we obtain
				\[
				\eta_{l-(j-2)}+\eta_{l-(i-2)}=n-2 = 2k+2
				\]
				and hence
				\[
				\begin{aligned}
					\eta_{l-(j-2)}&=1,\\
					\eta_{l-(i-2)}&=2k+1,
				\end{aligned}
				\]
				which is forbidden by definition of $\dbar$. \qedhere
			\end{itemize}
	\end{proof}
	
\begin{lemma}\label{2n-5i=j}
Let $2 \le i = j \le l+1$, and let $s_\alpha$ be a hook length outside the first column. If such an $s_\alpha$ exists, it satisfies
\[
h_{i,1} = 2\, h_{i,i} = 2\, s_\alpha.
\]
\end{lemma}

\begin{proof}
We distinguish two cases. If $\alpha \neq i-1$, then we are in the same situation as in Lemma~\ref{2n-5idiversoj} when $\alpha = i-1$. If $\alpha = i-1$, we obtain
\[
s_{i-1} = h_{i,i} = \l_{n-2} - 2 s_{i-1},
\]
and from this equation it follows that
\[
h_{i,1} = h_{1,i} = \l_{n-2} - s_{i-1} = 2 s_{i-1} = 2 h_{i,i}.\qedhere
\]
\end{proof}

\begin{lemma}\label{2n-5fuoriquadrato}
Let $2 \le j \le l+1$ and $l+2 \le i \le n-2$, and suppose that $s_\alpha$ is a hook length outside the first column. Then the existence of such an $s_\alpha$ leads to a contradiction.
\end{lemma}

\begin{proof}
		We have 
		\[
		s_{\alpha}=h_{i,j}=\l_{n-2}-s_{i-1}-s_{j-1}.
		\]
		From Lemma~\ref{2n-5hookbound} we have $s_{\alpha}<n-2$, hence $2\leq \alpha+1\leq l+1$.
		Since
		\[
		n-2>h_{\alpha+1,j}=\l_{n-2}-s_{j-1}-s_{\alpha}=s_{i-1},
		\]
		we are in the same conditions of Lemma~\ref{2n-5idiversoj}, which lead to a contradiction.
	\end{proof}

	All the preceding lemmas collectively show that any hook length $s_\alpha$ lying outside the first column of $Y_{\eta^*}$ must either coincide with the hook length of a cell in the first column, or be exactly half of the hook length of the corresponding diagonal cell.
Since every possible position outside the first column has now been analyzed, any additional hook would necessarily violate the explicit relations between hooks and odd parts.
Thus every potential refinement outside the first column is ruled out, and we may summarize the outcome in the following theorem.

\begin{theorem}\label{lambdanraffinabile-2n-5}
Let $\eta \in \dbar$, and let 
$
\lambda = S^c_{\eta^*}
$
be the partition associated with the (self-conjugate) Young diagram $Y_{\eta^*}$, where 
$S_{\eta^*}=\KN^{-1}(Y_{\eta^*})$.  
Then $\lambda$ is unrefinable.  
In particular,
\[
\lambda \in \overline{\mathcal{U}}_{T_{n,n-2k}}.
\]
\end{theorem}

\begin{proof}
By the previous lemmas, every hook length outside the first column either reproduces a hook in the first column or is forced to match half of a diagonal hook, and any other possibility leads to a contradiction. Hence $\lambda$ is unrefinable by Proposition~\ref{riconoscimentounref}. 
By Proposition~\ref{2n-5peso^*}, the partition $\lambda$ has weight $T_{n,n-2k}$.  
Corollary~\ref{cor:struct_lambda_n2k} ensures that $\lambda$ has maximal part $2n-5$ and exactly $n-3$ missing parts, which characterizes the elements of $\overline{\mathcal{U}}_{T_{n,n-2k}}$.  
\end{proof}

To conclude, observe that for $2 \le 2k \le n-4$ the previous results imply that the number of maximal unrefinable partitions of $T_{n,n-2k}$ coincides with
$
\# {{\mathbb{D}}}^{\, \text{o}}_{2k+2},
$
the count of odd-distinct partitions of $2k+2$.
Indeed,  in the case $2k=n-4$ we have \[{{\mathbb{D}}}^{\, \text{o}}_{2k+2} = \dbar \cup \{(1,2k+1)\},\] which corresponds bijectively to 
$\mup{T_{n,4}} = \overline {\mathcal U}_{T_{n,4}} \cup \{\widetilde{\tau}_n\}$, where the exceptional partition $\widetilde{\tau}_n$ of Tab.~\ref{tab:exceptions}  (which is not allowed in $\overline {\mathcal U}$ because it does not maximize the number of missing parts) is mapped to $(1,2k+1)$.
It can be seen by numerical examples that this particular partition  produces  nonconforming behaviours under the correspondence \(\KN^{-1}\): the associated diagram fails to satisfy the structural constraints that guarantee unrefinability (similarly to Example~\ref{controesempio}), and therefore must be removed from the bijection.

\section{nontriangular weight with maximal part $2n-4$}\label{sec:nontri4}
In this  final case we consider maximal unrefinable partitions of nontriangular weight $T_{n,n-2k+1}$, whose largest part equals $2n-4$,
and we aim to prove, for $8\leq 2k\leq n-2$, that 
\[\#\mup{T_{n,n-2k+1}} = 1+\# \mathbb{D}^{\,\text{e}}_{2k} = 1+\# \mathbb{D}_{k}.\] 
We establish this correspondence by constructing again a double embedding, presented in the two subsections that follow. In this case, the cardinality analysis of the sets of unrefinable partitions in the nontriangular case splits into two subcases depending on the value of $d$. Both scenarios lead to the same combinatorial conclusions regarding the correspondence with the set $\mathbb{D}_k$ of partitions of $k$ into distinct parts:

\begin{itemize}
    \item {case $d=3$ (or $2k=n-2$)}: in this setting, the partition $\widetilde\sigma_n$ of Tab.~\ref{tab:exceptions} belongs to  $\mup{}$. To proceed with the geometric characterization, we obtain $\overline{\mathcal{U}}$ by removing this partition, i.e.\ $\overline{\mathcal{U}} = \mup{} \setminus \{\widetilde\sigma_n\}$. It is then shown that the resulting set $\overline{\mathcal{U}}$ is equipotent to $\mathbb{D}_k$. Here, the ``$+1$'' term in the total count of unrefinable partitions is precisely accounted for by the inclusion of the previously removed partition $\widetilde\sigma_n$;
    
    \item {case $d > 3$}: in this case, the partition $\widetilde\sigma_n$ is not defined, so that $\overline{\mathcal{U}}=\mup{}$. However, in order to establish the correct bijection with the Young structures associated with distinct-part partitions of $k$, we identify and remove a specific partition \[\zeta_{n,k} \deq (1,2,\dots, n-k-3, n-k-1, \dots, n-3, n-2+k, 2n-4) \in \overline{\mathcal{U}}.\] It follows that $\overline{\mathcal{U}} \setminus \{\zeta_{n,k}\}$ is equipotent to $\mathbb{D}_k$.
\end{itemize}

Let proceed, as in the previous cases, by establishing the double embedding in two steps. It should be noted that the distinction between the cases $d=3$ and $d \neq 3$ will not be further utilized in the following derivations. Although the underlying sets are defined differently, the geometric procedures and the logic of the proofs remain valid in every instance; the counting arguments effectively compensate for these differences as detailed above, ensuring the same numerical results in both scenarios.

Let $\lambda = (\lambda_1, \dots, \lambda_t)$ be a partition such that 
\[
\lambda \in 
\begin{cases} 
\,\overline{\mathcal{U}}_{T_{n,n-2k+1}} & \text{if } d=3, \\ 
\,\overline{\mathcal{U}}_{T_{n,n-2k+1}} \setminus \{\zeta_{n,k}\} & \text{if } d > 3. 
\end{cases}
\]

As for the structural properties of $Y_{S_\lambda}$, one can easily note that all results from Lemma~\ref{lemLR} to Lemma~\ref{d2z+1} proven in the triangular case hold also in this case. Therefore, as in the previous Sec.~\ref{sec:tri}, the Young diagram associated via the $\KN$ transformation exhibits the geometric pattern of a main square region together with a single additional column (the extra column) placed immediately immediately to the right of the square. The only significant difference compared with the triangular case analyzed in Sec.~\ref{sec:tri} is that the weight observed along the main diagonal is different. This structural variation is formalized in the following result.

\begin{proposition}
    Let $\lambda \in \overline{\mathcal{U}}_{T_{n,n-2k+1}}$. Then the total number of cells in $Y_{S_\lambda}$ is 
    \[
    \#Y_{S_\lambda} = 2n - 4 + 2k.
    \]
    Consequently, the sum of the hook lengths along the main diagonal is $2n - 4 + 2k$.
\end{proposition}

\begin{proof}
    The total number of cells in the diagram can be computed as the sum of the number of cells in each row:
    \[
    \begin{aligned}
        \#Y_{S_\lambda} &= \sum_{i=1}^{n-2} \#R_i \\
        &= \sum_{i=1}^{n-2} h_{i,1} - (\#C_1 - i) \\
        &= \sum_{i=1}^{n-2} h_{i,1} - \sum_{i=1}^{n-2} \#C_1 + \sum_{i=1}^{n-2} i \\
        &= T_{n,n-2k+1} - (n-2)(n-2) + T_{n-2} \\
        &= \frac{n(n+1)}{2} - (n-2k+1) - (n-2)^2 + \frac{(n-2)(n-1)}{2} \\
        &= \frac{1}{2}(4n - 8 + 4k) \\
        &= 2n - 4 + 2k.
    \end{aligned}
    \]
\end{proof}

Note that this proves the first part of the embedding, namely that to each partition $\lambda$ of $\overline{\mathcal{U}}$ as defined above, there corresponds a distinct-part partition of weight $k$. 

To conclude, we prove the converse by considering $\eta = (\eta_1, \dots, \eta_l) \in \mathbb{D}_k$ and defining $\eta^* = (\eta_1, \dots, \eta_l, n-2)$ and $2\eta^* = (2\eta_1, \dots, 2\eta_l, 2n-4)$. Under these definitions, the Young diagram $Y_{2\eta^*}$ satisfying the following conditions
\begin{align*}
h_{1,1} &= 2n - 4, \nonumber \\
h_{i,i} &= 2\eta_{\,l-(i-2)}, \qquad 2 \le i \le l+1, \nonumber \\
a(c_{i,i}) &= l(c_{i,i}) + 1, \qquad 1 \le i \le l+1, %\label{equazionecostruzione'}
\end{align*}
defines an unrefinable partition of weight $T_{n,n-2k+1}$.

As above, many results remain identical to those presented in Sec.~\ref{sec:tri}. In particular, the properties from Lemma~\ref{rowsColsY} to Lemma~\ref{block2symmetry} continue to hold in this context. Consequently, we only need to establish the following proposition regarding the weight calculation.

\begin{proposition}
Let $\eta \in \mathbb D_k$ and  let $\lambda = S^c_{2\eta^*}$ be the partition associated to the Young diagram $Y_{{2\eta^*}}$, where $S_{2\eta^*} = \KN^{-1}(Y_{2\eta^*})$. 
Then $\lambda$ satisfies
\[
\lambda \vdash T_{n,n-2k+1},
\]
where $8 \le 2k \le n-2$.
\end{proposition}

\begin{proof}
The weight of $\lambda$ is equal to the sum of the hook lengths in the first column of $Y_{\eta^*}$. Hence, we have
\[
\begin{aligned}
\sum_{i=1}^{n-2} h_{i,1} 
&= h_{1,1} + \sum_{i=2}^{n-2} h_{i,1}  \\
&= (2n-4) + \sum_{i=2}^{n-2} \bigl( l(c_{i,1}) + 1 + a(c_{i,1}) \bigr) \\
&= (2n-4) + \sum_{i=1}^{n-3} i + \sum_{i=1}^{l}  2\eta_i \\
&= 2n-4 + \frac{(n-3)(n-2)}{2} + 2k\\
&= T_{n,n-2k+1}. 
\end{aligned}
\]
\end{proof}

In order to conclude the proof of the unrefinability of $\lambda$, we proceed by following the same reasoning developed in Sec.~\ref{subsec:unref}, specifically relying on Lemmas~\ref{salphaidiversoj}, \ref{salphai=j}, \ref{lemmaboh}, and \ref{lemmaboh2}. The proofs of Lemma~\ref{lemmaboh} and Lemma~\ref{lemmaboh2} are identical. On the other hand, the proofs of Lemma~\ref{salphaidiversoj} and Lemma~\ref{salphai=j} require only straightforward adaptations to the current setting and are reported here for the sake of completeness.

\begin{proof}[Proof (of Lemma~\ref{salphaidiversoj})]
We proceed as in the proof of Lemma~\ref{salphaidiversoj} of  Sec.~\ref{subsec:unref} and we obtain:
\begin{equation*}%\label{eq:boh}
s_{\alpha} = \eta_{l-(i-2)} + \eta_{l-(j-2)}.
\end{equation*}

Now, let us assume $\alpha+1\neq i,j$. Then, since $\eta\vdash k$, we have
		$
		s_{\alpha}=\eta_{l-(i-2)}+\eta_{l-(j-2)}\leq k < n-2
		$.
		Now, we have
		$
		s_\alpha = n-2-\eta_{l-(\alpha-1)},
		$
		therefore 
		\[n-2=\eta_{l-(i-2)}+\eta_{l-(j-2)}+\eta_{l-(\alpha-1)}\leq k < n-2,
		\]
		a contradiction.\\
		Without loss of generality, let us now assume $\alpha=i-1$. Then we have 
		\begin{equation*}%\label{equazioneidiversoj}
			n-2=\eta_{l-(j-2)}+2\eta_{l-(i-2)}.
		\end{equation*}
		Computing
		\[
		\eta_{l-(j-2)}+\eta_{l-(i-2)}=\eta_{l-(j-2)}+\frac{n-2-\eta_{l-(j-2)}}{2}=\frac{n-2+\eta_{l-(j-2)}}{2},
		\]
		since $\eta\vdash k$, we obtain
		\[
		n-2+\eta_{l-(j-2)}\leq 2k\leq n-2,
		\]
		which is again a contradiction.
\end{proof}

\begin{proof}[Proof (of Lemma~\ref{salphai=j})]
From the hypotheses we have
		\begin{equation*}
			s_{\alpha} = h_{i,i} = 2\eta_{l-(i-2)}.
		\end{equation*}
As in the original proof, we proceed by distinguishinging three possibile cases:
$s_{\alpha} < n-2$,
$s_{\alpha} = n-2$, or
 $s_{\alpha} > n-2$.
 
 In the first case, we proceed in the same way as in the original proof and obtain the exact same conclusion. 
			Assuming  $s_{\alpha} = n-2$, from Eq.~\eqref{equazionei=jsalpha} we have
			\[
			2\eta_{l-(i-2)} = s_{\alpha} = n-2 \geq 2k,
			\]
			which is a contradiction since no part in $\eta$ can be larger than $k$.
			Finally, assume $s_{\alpha} > n-2$, hence $s_{\alpha} \geq n-1> 2k$. 
			Therefore, $\eta_{l-(i-2)} \geq k$, again a contradiction.
\end{proof}

From the previous considerations we obtain the last main contribution of this work.
\begin{theorem}\label{thm:2n4_unrefinable}

Let $\eta \in \mathbb{D}_k$, and let 
$
\lambda = S^c_{\eta^*}
$
be the partition associated with the  Young diagram $Y_{\eta^*}$, where 
$S_{\eta^*}=\KN^{-1}(Y_{\eta^*})$.  
Then $\lambda$ is unrefinable.  
In particular,
\[
\lambda \in 
\begin{cases} 
\,\overline{\mathcal{U}}_{T_{n,n-2k+1}} & \text{if } d=3, \\ 
\,\overline{\mathcal{U}}_{T_{n,n-2k+1}} \setminus \{\zeta_{n,k}\} & \text{if } d > 3. 
\end{cases}
\]

\end{theorem}

The previous results imply that, for $8 \leq 2k \leq n-2$,  the number of maximal unrefinable partitions of $T_{n,n-2k+1}$ coincides with $1 + \# \mathbb{D}_k$. 
As previously detailed, this correspondence is  a bijection once the exceptional cases are handled. In the case $d=3$, the ``$+1$'' term is accounted for by the partition $\widetilde{\sigma}_n$, which is unrefinable but stands outside the main family $\overline{\mathcal{U}}$. Conversely, for $d > 3$, the ``$+1$'' term is represented by the partition $\zeta_{n,k} \in \overline{\mathcal{U}}$, which must be set aside in order to establish the equipotence between the remaining set and $\mathbb{D}_k$. 

There is exactly one obstruction to the above bijection, represented by the partition $\zeta_{n,k}$, whose image under $\KN$ is the one-part partition $(k)$, excluded by definition (as in the requirements established in Sec.~\ref{sec:part}). We conclude the paper by showing  with an example the necessity of this omission when $d > 3$. 

\begin{example}
Let us consider the partition
\[
\zeta_{19,6} = (1,2,\dots,10,12,13,\dots,16,23,34) \in \overline{\mathcal U}_{T_{19,8}}.
\]
The Young diagram shown in Fig.~\ref{fig:young2} highlights, in red, the cell whose hook length corresponds, under the Keith--Nath correspondence, to the partition $(6)$. 
Since $(6)$ is not a partition into distinct parts, the partition $\zeta_{19,6}$ must be excluded from the correspondence described in Theorem~\ref{thm:2n4_unrefinable}.

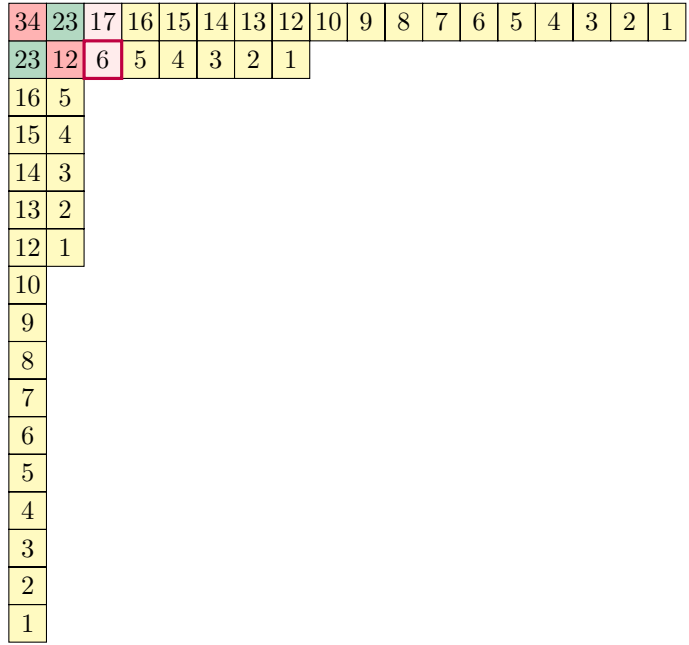
\begin{figure}
\begin{tikzpicture}
	\draw[color=black, fill=yellow!30] (0,-8.5) rectangle (0.5,-8) rectangle (0,-7.5) rectangle (0.5,-7) rectangle (0,-6.5) rectangle (0.5,-6) rectangle (0,-5.5) rectangle (0.5,-5) rectangle (0,-4.5) rectangle (0.5,-4) rectangle (0,-3.5) rectangle (0.5,-3) rectangle (0,-2.5) rectangle (0.5,-2) rectangle (0,-1.5) rectangle   (0.5,-1);
	
	\draw[color=black, fill=yellow!30]  (1.5,0) rectangle (2,-0.5)  rectangle (2.5,0) rectangle (3,-0.5) rectangle (3.5,0) rectangle (4,-0.5)  rectangle (4.5,0) rectangle (5,-0.5)  rectangle (5.5,0) rectangle (6,-0.5) rectangle (6.5,0) rectangle (7,-0.5) rectangle (7.5,0) rectangle (8,-0.5) rectangle (8.5,0) rectangle (9,-0.5);
	
	\draw[color=black, fill=yellow!30] (1,-3.5) rectangle (0.5,-3) rectangle (1,-2.5)rectangle (0.5,-2) rectangle (1,-1.5)rectangle (0.5,-1);
	
	\draw[color=black, fill=yellow!30] (1.5,-1) rectangle (2,-0.5)rectangle (2.5,-1) rectangle (3,-0.5)rectangle (3.5,-1) rectangle (4,-0.5);
	
	\draw[color=black, fill=ForestGreen!30] (0,-1) rectangle (0.5, -0.5) rectangle (1,0);
	
	\draw[color=black, fill=red!30] (0,0) rectangle (0.5,-0.5) rectangle (1,-1);
	
	\draw[color=black, fill=pink!30] (1,0) rectangle (1.5,-0.5);
	
	\draw[color=purple, fill=pink!30, line width=1.2pt] (1.5,-0.5) rectangle (1,-1);

    \node at (0.25,-0.25) {$34$};
    \node at (0.25,-0.75) {$23$};
    \node at (0.25,-1.25) {$16$};
    \node at (0.25,-1.75) {$15$};
    \node at (0.25,-2.25) {$14$};
    \node at (0.25,-2.75) {$13$};
    \node at (0.25,-3.25) {$12$};
    \node at (0.25,-3.75) {$10$};
    \node at (0.25,-4.25) {$9$};
    \node at (0.25,-4.75) {$8$};
    \node at (0.25,-5.25) {$7$};
    \node at (0.25,-5.75) {$6$};
    \node at (0.25,-6.25) {$5$};
    \node at (0.25,-6.75) {$4$};
    \node at (0.25,-7.25) {$3$};
    \node at (0.25,-7.75) {$2$};
    \node at (0.25,-8.25) {$1$};

    \node at (0.75,-0.25) {$23$};
    \node at (1.25,-0.25) {$17$};
    \node at (1.75,-0.25) {$16$};
    \node at (2.25,-0.25) {$15$};
    \node at (2.75,-0.25) {$14$};
    \node at (3.25,-0.25) {$13$};
    \node at (3.75,-0.25) {$12$};
    \node at (4.25,-0.25) {$10$};
    \node at (4.75,-0.25) {$9$};
    \node at (5.25,-0.25) {$8$};
    \node at (5.75,-0.25) {$7$};
    \node at (6.25,-0.25) {$6$};
    \node at (6.75,-0.25) {$5$};
    \node at (7.25,-0.25) {$4$};
    \node at (7.75,-0.25) {$3$};
    \node at (8.25,-0.25) {$2$};
    \node at (8.75,-0.25) {$1$};

    \node at (0.75,-0.75) {$12$};
    \node at (0.75,-1.25) {$5$};
    \node at (0.75,-1.75) {$4$};
    \node at (0.75,-2.25) {$3$};
    \node at (0.75,-2.75) {$2$};
    \node at (0.75,-3.25) {$1$};

    \node at (1.25,-0.75) {$6$};
    \node at (1.75,-0.75) {$5$};
    \node at (2.25,-0.75) {$4$};
    \node at (2.75,-0.75) {$3$};
    \node at (3.25,-0.75) {$2$};
    \node at (3.75,-0.75) {$1$};
    \end{tikzpicture}
    	\caption{Young diagram for the partition $\zeta_{19,6}$ illustrating the nonconforming behavior of $\KN^{-1}$.}
    \label{fig:young2}
\end{figure}
\end{example}

\section*{Conclusions}

In this paper we considered a geometric approach to  study unrefinable partitions,
based on the Keith--Nath correspondence between numerical sets and Young diagrams.
The central results are based on a hook-length criterion that characterizes unrefinability in purely
combinatorial and diagrammatic terms, without appealing to enumerative arguments.
Using this framework, we revisited and unified several known correspondences between
maximal unrefinable partitions and partitions into distinct parts.

Although the triangular and nontriangular cases are treated separately, they follow a
remarkably parallel structure.
For the reader's convenience, we summarize in Tab.~\ref{tab:triangular-vs-nontriangular} the main features of the two settings,
highlighting their close analogy.

The geometric viewpoint introduced here suggests several directions for future work.
Possible extensions include the study of nonmaximal unrefinable partitions, refinability
phenomena under controlled deformations of the diagram, and the interaction between
hook-length symmetries and other classes of constrained partitions.
More generally, the results indicate that the geometry of Young diagrams provides a
natural and flexible language for understanding refinement properties of partitions.

\begin{table}
	\centering
	%\small 
	\begin{tabular}{c|c||c|c}
		\hline
		& \textbf{Triangular case} 
		& \textbf{Nontriangular case} 
		& \textbf{Nontriangular case} \\
		&  & $\l_t = 2n-5$ & $\l_t = 2n-4$ \\
		\hline\hline
		Weight 
		& $T_n$ 
		& $T_{n,n-2k}$ 
		& $T_{n,n-2k+1}$ \\
		\hline
		Upper-left region  
		& square +
		& square 
		& square + \\
		of the diagram
		&  one extra column 
		&  
		&  one extra column \\
		\hline
		Symmetry 
		& quasi-symmetric 
		& self-conjugate 
		& quasi-symmetric  \\
		\hline
		Diagonal encodes 
		& even-distinct parts 
		& odd-distinct parts 
		& even-distinct parts \\
		\hline
		Target partitions 
		& $\mathbb{D}_k = \mathbb{D}^{\,\text{e}}_{2k}$ 
		& $\mathbb{D}^{\,\text{o}}_{2k+2}$ 
		& $\mathbb{D}_k = \mathbb{D}^{\,\text{e}}_{2k}$ \\
		\hline
	\end{tabular}
	
\medskip
\caption{Comparison between the triangular and nontriangular cases.}
\label{tab:triangular-vs-nontriangular}
\end{table}

\section*{Statements and Declarations}
The authors have no relevant financial or nonfinancial interests to disclose. R. Aragona, L. Campioni and R. Civino wrote the main manuscript text. All authors reviewed the manuscript. Data availability statements: no datasets were generated or analyzed during the current study,and therefore no data are available to be shared.

\bibliographystyle{amsalpha}
\bibliography{sym2n_ref}

\end{document}